\documentclass[11pt]{amsart}
\usepackage{amssymb,amscd}
\usepackage{pb-diagram}
\usepackage{verbatim}
\usepackage{rotating}
\usepackage{graphicx}
\usepackage{amssymb,amscd,enumerate}
\usepackage{xcolor}

\newtheorem{thm}{Theorem}[section]
\newtheorem*{thmnn}{Theorem}
\newtheorem{lem}[thm]{Lemma}
\newtheorem{cor}[thm]{Corollary}

\newtheorem{prop}[thm]{Proposition}

\newtheorem{defn}[thm]{Definition}

\newtheorem{question}[thm]{Question}
\newtheorem{conj}[thm]{Conjecture}

\numberwithin{equation}{section}
\numberwithin{figure}{section}

\newcommand{\lnk}{\ell k}

\newcommand{\Id}{\operatorname{Id}}

\newcommand{\Span}{\operatorname{span}}

\newcommand{\Int}{\displaystyle\int}
\newcommand{\Lim}{\displaystyle\lim}

\newcommand{\Z}{\mathbb{Z}}

\newcommand{\T}{\mathbb{T}}
\newcommand{\N}{\mathbb{N}}
\newcommand{\C}{\mathbb{C}}
\newcommand{\Q}{\mathbb{Q}}
\newcommand{\K}{\mathbb{K}}
\newcommand{\R}{\mathbb{R}}

\newcommand{\Bl}{\mathcal{B}\ell}

\newcommand{\A}{\mathcal{A}}
\newcommand{\AC}{\mathcal{AC}}

\newcommand{\FF}{\mathcal{F}}

\newcommand{\CC}{\mathcal{C}}

\def\p{\partial}
\def\bdry{\partial}

\newcommand{\Gl}{\operatorname{Gl}}

\newcommand{\rank}{\operatorname{rank}}

\newcommand{\im}{\operatorname{im}}
\newcommand{\into}{\hookrightarrow}

\newcommand{\ot}{\leftarrow}
\newcommand{\Sum}{\displaystyle\sum}

\title[Derivatives of slice knots]{Cut open null-bordisms and derivatives of slice knots}

\author{Tim Cochran$^{\dag}$, Christopher William Davis	}

\thanks{$^{\dag}$The first author was partially supported by National Science Foundation grant DMS-1309081 and by a grant from the Simons Foundation (304603).}

\begin{document}

\begin{abstract}
In the 60's Levine proved that if $R$ is a slice knot, then on any genus $g$ Seifert surface for $R$ there is a  $g$ component link $J$, called a derivative of $R$, on which the Seifert form vanishes.  Many subsequent obstructions to $R$ being slice are given in terms of slice obstructions of $J$.  Many of these obstructions can be derived from a 4-manifold called a null-bordism.  Recently the authors proved that that it is possible for $R$ to be slice without $J$ being slice, disproving a conjecture of Kauffmann from the 80's.  In this paper we cut open these null-bordisms in order to derive new obstructions to being the derivative of a slice knot.  As a proof of the strength of this approach we re-derive a signature condition due to Daryl Cooper.  Our results also apply to doubling operators, giving new evidence for their weak injectivity.  We close with a new sufficient condition for a genus 1 algebraically slice knot to be $1.5$-solvable.
\end{abstract}

\maketitle

\section{Introduction}\label{sec:Introduction}
 A \textbf{knot} $K$ is an isotopy class of smooth embeddings of an oriented $S^1$ into $S^3$.  In \cite{FoMi} Fox and Milnor asked which knots might bound smoothly embedded disks in the 4-ball.  Such knots are called \textbf{slice}.  The question of what knots are slice has been been the subject of intense study ever since.  
 
 While not every knot is a slice knot, every knot is the boundary of a compact oriented embedded surface in $S^3$ called a \textbf{Seifert surface}. In the late 1960's, Jerome Levine began a program aimed towards deciding if a given knot is slice by studying one of its Seifert surfaces. This program succeeded for higher-dimensional knots ($S^{2n-1}\hookrightarrow S^{2n+1}, n>1$)~\cite{L5}. Recently the authors showed that for $n=1$  Levine's philosophy of focussing on the  Seifert surface has unexpected flaws ~\cite{CD1}. The purpose of the present paper is to quantify the extent of these flaws and to suggest a modification of this philosophy.
 
These questions have usually been studied in a more algebraic context since there is an abelian group whose identity element is the set of slice knots. Specifically, recall that  $K_0\hookrightarrow S^3\times\{0\}$ is \textbf{concordant} to $K_1\hookrightarrow S^3\times\{1\}$ if there exists a properly, smoothly embedded annulus in $S^3\times [0,1]$ that restricts on its boundary to the given knots. The set of concordance classes, $\CC$, is an abelian group under the operation of connected sum.  The identity in this group is the equivalence class of the trivial knot.  Inverses are given by the reverse of the mirror image.  A  filtration of $\mathcal{C}$, called the \textbf{$n$-solvable filtration}, has been considered \cite{COT}\cite[p.1423]{CHL3}\cite{CHL6}:
$$
\cdots \subseteq \mathcal{F}_{n+1} \subseteq \mathcal{F}_{n.5}\subseteq\mathcal{F}_{n}\subseteq\cdots \subseteq
\mathcal{F}_1\subseteq \mathcal{F}_{0.5} \subseteq \mathcal{F}_{0} \subseteq \mathcal{C}.
$$
 This filtration has provided a convenient framework for many recent advances in the study of knot concordance. 

We now discuss Levine's strategy in more detail. Levine considered the \textbf{Seifert form}, a bilinear form $\beta_F:H_1(F)\times H_1(F)\to \Z$ given by $\beta_F([x],[y])= \lnk(x,y^+)$ where $x,y$ are oriented simple closed curves on $F$, $y^+$ denotes the result of pushing $y$ off of $F$ in the positive normal direction, and $\lnk$ denotes the linking number. He proved that if $K$ is a slice knot with slice disk $\Delta$ then for any Seifert surface $F$, the  Seifert form is \textbf{metabolic}, meaning that there exists  half-rank summand $\Z^g\subseteq H_1(F)$ on which $\beta_F$ is identically zero  ~\cite[Lemma 2]{L5}. This is equivalent to saying that if $K$ is a slice knot then for any genus $g$ Seifert surface $F$ there exists a nonseperating $g$-component link $J=\{d_1,...,d_g\}$ on $F$, so that $\lnk(d_i,d_j^+)=0$ for all $i,j$. We will call this a \textbf{derivative of $K$ associated to $\Delta$}~\cite{CHL7}.

Any knot whose Seifert form is metabolic is called an \textbf{algebraically slice knot}. In higher dimensions Levine proved that an algebraically slice knot is a slice knot. In the classical dimension there are additional obstructions to $K$ being a slice knot ~\cite{CG1,CG2}. Significantly, these obstructions can also be expressed in terms of (the Seifert form of !) a derivative link. Motivation for studying derivatives was also given by the following elementary well-known result.    

\begin{prop}\label{prop:easyfact} A derivative, $J$,  of $K$ is a slice link $\Longrightarrow$ $K$ is a slice knot via a slice disk $\Delta$ to which $J$ is associated.
\end{prop}

This follows from using two copies of the disjoint slice disks for the components of $J$ to perform ambient surgery on $F$, altering it to a slice disk for $K$.  What it means for a derivative to be ``associated'' is an algebraic consequence of this construction.  See Definition~\ref{def:assocLagrangian}. Thus,  hope remained that Levine's strategy was sound. In fact a converse of Proposition~\ref{prop:easyfact} was conjectured. 
 
\begin{conj}\label{conj:Kauffman} ~(c.f. \cite[p.226]{OnKnots})

\noindent  $K$ is a slice knot via a slice disk $\Delta$ to which $J$ is associated $\Longrightarrow$ $J$ is a slice link (hence $\Longleftrightarrow$) .
\end{conj}

Proposition~\ref{prop:easyfact} and Conjecture~\ref{conj:Kauffman} translate over to the setting of the solvable filtration.  

\begin{prop}\label{prop:doublingraisessolvability}~\cite[Cochran-Orr-Teichner,  Theorem 8.9]{COT} A derivative, $J\in\mathcal{F}_{n}$  $\Longrightarrow$ $K\in \mathcal{F}_{n+1}$  (via an $(n+1)$-solution to which $J$ is associated). Here $n$ is a non-negative integer or half-integer.
\end{prop}

\begin{conj}\label{conj:Kauffmannsolv} 
\noindent  $K\in \mathcal{F}_{n+1}$ (via an $(n+1)$-solution to which $J$ is associated) and $\Delta_K(t)\neq 1$  $\Longrightarrow$ $J\in \mathcal{F}_{n}$.
\end{conj}

In \cite{CD1} The authors exhibit a counterexample to Conjecture \ref{conj:Kauffman} which simultaneously disproves Conjecture~\ref{conj:Kauffmannsolv}. The main goal of this paper is to establish a positive result that points to a correct version of these conjectures.  Instead of beginning with our most technical results we give easier to understand implications when $n=0.5$.  Recall that a knot is $0.5$-solvable if and only if it is algebraically slice.  

\subsection{Derivatives of $1.5$-solvable genus $1$ knots.}

The counterexample of \cite{CD1} consists of a genus one \textit{slice} knot $K$ where $\Delta_K(t)= (mt-(m+1))((m+1)t-m)\neq 1$ (with $m=1$)  and where the derivative $J$ associated to the slice disk was \textit{not} algebraically slice, but rather was a sum of two cabled knots $S_{(m,1)}\#-S_{(m+1,1)}$ (See Figure~\ref{fig:cable}) and hence whose algebraic concordance type satisfies
$$
[J]=[S_{(m,1)}]-[S_{(m+1,1)}].
$$
The results presented in this subsection suggest that this is not an accident, rather, we conjecture that genus one algebraically slice knots admitting such a derivative are precisely those which are $1.5$-solvable.

\begin{conj}\label{conj:genus1} Suppose that $K$ admits a genus one Seifert surface and a derivative $J$ (consequently $J$ is a knot and $\Delta_K(t)\doteq (mt-(m+1)((m+1)t-m)$.) Suppose $m\notin{0,-1}$.  

The algebraic concordance type of $J$ satisfies
$
[J]=[S_{(m,1)}]-[S_{(m+1,1)}]
$
 for some knot $S$, if and only if $K$ is $1.5$-solvable (via a $1.5$-solution to which $J$ is associated).
\end{conj}

\begin{figure}[t]
\setlength{\unitlength}{1pt}
\begin{picture}(165,70)
\put(0,0){\includegraphics[height=70pt]{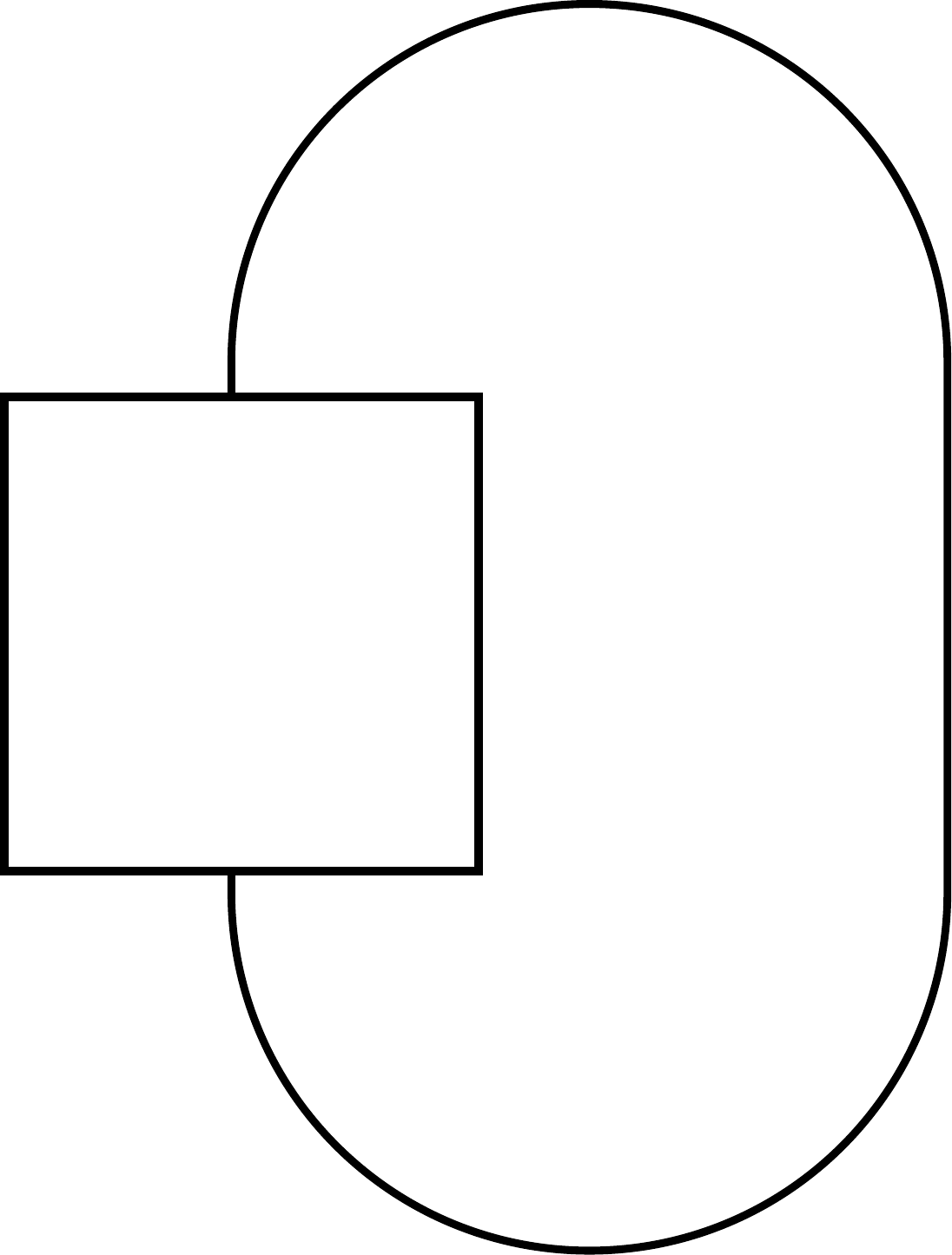}}
\put(100,0){\includegraphics[height=70pt]{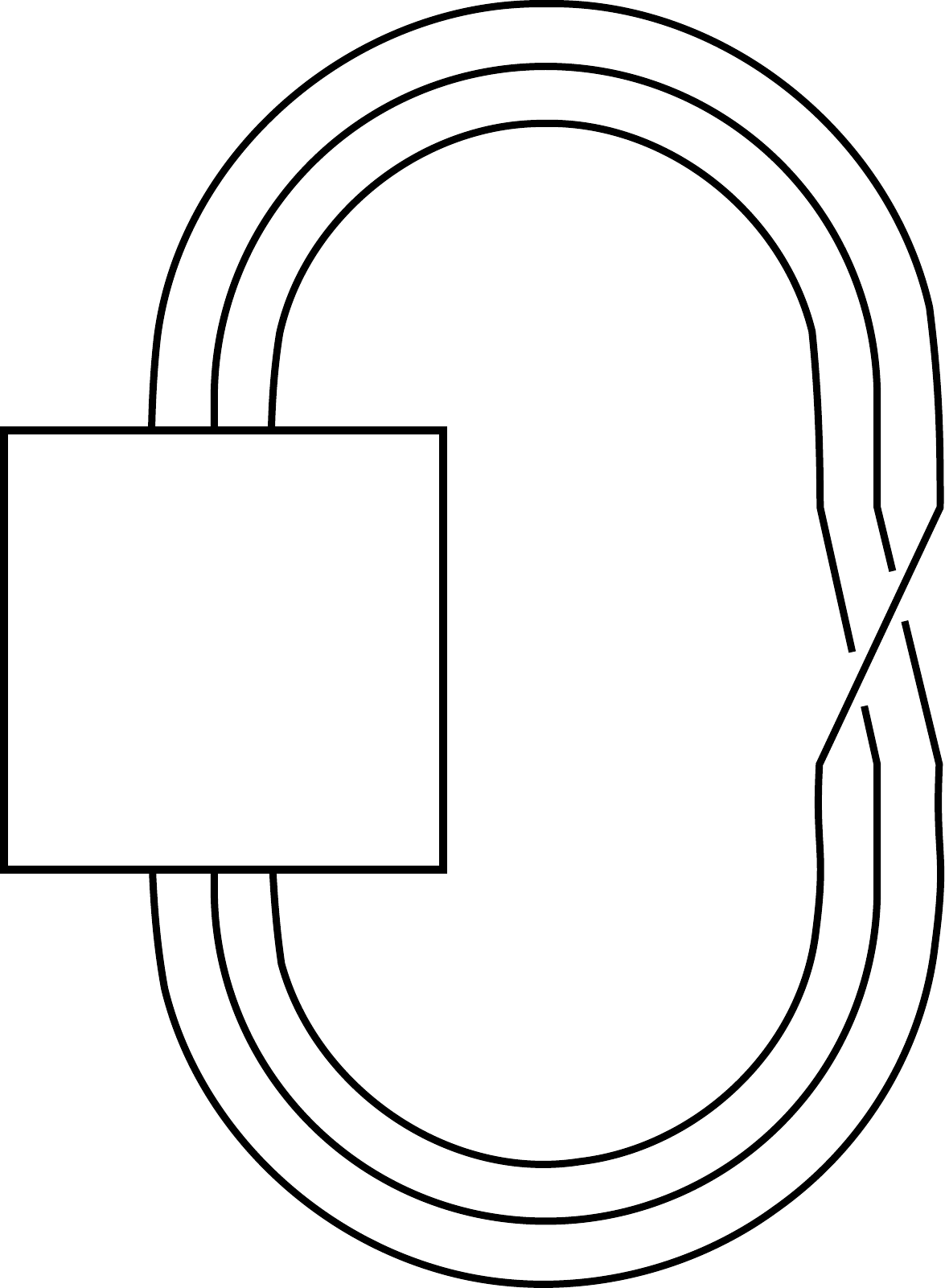}}
\put(10,30){S}
\put(110,30){S}
\end{picture}
\caption{A knot $S$ and its $(3,1)$-cable}\label{fig:cable}
\end{figure}

In Section~\ref{sect:suff} we prove the sufficiency of the condition conjectured above:

\begin{thm}\label{thm:sufficiency} Suppose that $K$ admits a genus one Seifert surface and a derivative $J$ (consequently $J$ is a knot and $\Delta_K(t)\doteq (mt-(m+1))((m+1)t-m)$. Suppose $m\notin \{0, -1\}$.  

If  the algebraic concordance type of $J$ satisfies
$
[J]=[T_{(m,1)}]-[T_{(m+1,1)}]
$
 for some knot $T$, then $K$ is $1.5$-solvable (via a $1.5$-solution to which $J$ is associated.)
\end{thm}

We next discuss the evidence for reverse implication.  Given any knot $J$ there is a function called the \textbf{Levine-Tristram} signature function $\sigma_J:\T\to \Z$.  ($\T$ is unit circle in the complex plane.)  After renormalizing at the discontinuities, the rule sending $J\mapsto \sigma_J$ is a  well defined homomorphism on the algebraic concordance group $\AC = \frac{\CC}{\FF_{0.5}}$.  The kernel of the map $J\mapsto \sigma_J$ is the torsion subgroup of $\AC$.  We prove the following result.

\begin{thm}\label{thm:genusonesign} Suppose that $K$ admits a genus one Seifert surface, $F$, admitting a derivative $J$ (consequently $J$ is a knot and  $\Delta_K(t)\doteq (mt-(m+1))((m+1)t-m)$). Suppose $m\notin \{0, -1\}$.  

 If $K\in \mathcal{F}_{1.5}$, in particular if $K$ is slice, and $J$ is an associated derivative, then  the Levine-Tristram signature function of $J$  satisfies that 
$$
\sigma_{J}(\omega^k) = g(\omega^{m+1})-g(\omega^m)
$$
for some $k\in \N$ and a function $g:\T\to \R$.  Moreover $g$ can be chosen to be continuous away from the roots of some polynomial where $g$ has at worst jump discontinuities.  
\end{thm}

The operation of cabling has a well understood effect on algebraic concordance (and so on the signature function):
$$
\sigma_{J_{(k,1)}}(\omega) = \sigma_J(\omega^k).
$$
Thus, Theorem~\ref{thm:genusonesign} should suggest that if $J$ is a derivative associated to a  ($1.5$)-solution then some cable of $J$ is algebraically concordant to a difference of cables of some knot:
$
[J_{(k,1)}] = [S_{(m,1)}]-[S_{(m+1,1)}].
$

As a demonstration of the strength of Theorem~\ref{thm:genusonesign} we recover Daryl Cooper's famous signature condition on the derivatives of a slice knot.  (Proposition~\ref{prop:CooperCondition1})

\begin{thm}[Chapter 1, Corollary 3.14 of \cite{CooperThesis}, See also Theorem 2 of \cite{GL2}]\label{Thm: Cooper sign 1}
If $K$ is a genus one slice knot with $\Delta_K(t)=((m+1)t-m)(mt-(m+1))$ then on any genus one Seifert surface for $K$ there is a derivative $J$ which satisfies that for all $p, c\in \N$  such that $(p,m)=(p,m+1)=(c,p)=1$ it follows that 
$$
\displaystyle \sum_{\ell=1}^{r} \sigma_J\left( e^{2\pi i \frac{c ((m+1) \overline{m})^\ell}{p}} \right) = 0
$$
where $\overline{m}$ is an inverse to $m$ mod p and $r$ is the order of $(m+1)\overline{m}$ in $\Z/p$.
\end{thm}

\subsection{Cut open null-bordisms}\label{subsect:cut open}

The results of the preceding subsection follow by studying a $4$-manifold called a \textbf{null-bordism}, bounded by a derivative associated to a slice disk.  See \cite[Definition 8.3 and Proposition 8.4]{CHL7} and Proposition~\ref{prop:NullBordismDeriv} of this paper.  This 4-manifold has many of the properties of a slice disk complement for $J$, except that $H_1(M_J)\to H_1(V)\cong \Z$ is the zero map.  Our key observation is that if one lets $\Sigma$ be a 3-manifold dual to a generator of $H_1(V)$ and $N$ be a product neighborhood of $\Sigma$, one gets what we will call a \textbf{cut open null-bordism} $W=V-N$ (See also Definition~\ref{defn: cut open}) with boundary given by $M_J$ together with two copies of $\Sigma$: $\Sigma^+$ and $\Sigma^-$.  The latter has its orientation reversed.  This cobordism has many properties in common with a slice disk complement, for example, the map $H_1(M_J)\to H_1(W)$ is often non-trivial.  A key difference is that its boundary now has three components, $\bdry W = M_J\sqcup \Sigma^+ \sqcup \Sigma^-$.  The following technical result reveals that from the point of view of the techniques of \cite{COT} cut open null-bordisms are very much like slice disk complements.

\begin{thm}\label{thm:sampleMainTheorem}[See Proposition \ref{prop:IsectnForm} and Corollary \ref{cor:IsectnFormCover}]
Suppose that $M_J$ bounds a null-bordism $V$, $\Sigma$ is a 3-manifold dual to the generator of $H_1(V)$ with neighboorhood $N$, and $W=V-N$.   Let $\phi:\pi_1(V)\to \Gamma$ be a homomorphism to a PTFA group.  Assume that $\rank(H_1(M_J;\K(\Gamma))) = \beta_1( M_J)-1$.    Let $A$ be the image of  $\pi_1(W)\to \pi_1(V)\to \Gamma$ then $\dfrac{H_2(W)}{H_2(\bdry W)} = \dfrac{H_2(W;\K(A))}{H_2(\bdry W;\K(A))} = 0$.
\end{thm}

  What it means for  a group to be PTFA is recalled in Section~\ref{sect:bordism}.  Importantly if $\Gamma$ is PTFA then $\Q[\Gamma]$ embeds in a skew-field of fractions $\K(\Gamma)$. See\cite[Chapter 2]{ste}) and \cite[Proposition 32]{C}.)

  Observe that the cut-open null-bordism $W$ is a fundamental domain for the infinite cyclic cover $\widetilde{V}$ of $V$.   If $t:\widetilde{V}\to \widetilde{V}$ is the generator of the deck group, then $t$ sends $\Sigma^+\subseteq \bdry W$ to$\Sigma^-$.    At this point the reader should think that the correct version of Conjecture~\ref{conj:Kauffman} concludes not that $J$ is a slice link but that, loosely speaking, $M_J$ is cobordant via $W$ to $\Sigma^+\sqcup \Sigma^- = (t_*-\Id)\Sigma$.

Next we discuss a refinement of Theorem~\ref{thm:sampleMainTheorem} to the solvable filtration.  If $K$ is $(n+1)$-solvable, then $J$ admits an $(n+1)$-null-bordism, $V$.  (See \cite{CHL7} and Definition~\ref{defn:NullBordism}).  We can cut open $V$ to get a fundamental domain $W$ for the infinite cyclic cover $\widetilde{V}$.  Indeed, a theorem similar to Theorem~\ref{thm:sampleMainTheorem} holds (Theorem~\ref{thm:nIsectnForm}) when $V$ is an $(n+1)$-null-bordism and $\Gamma^{(n+1)}_r=0$.  ($\Gamma^{(n+1)}_r$ is the $(n+1)$'st term in the rational derived series, see Section~\ref{sect:bordism}).  A slightly weaker result (Theorem~\ref{thm:n.5IsectionForm}) holds when $V$ is an $n.5$-null-bordism and $\Gamma^{(n+1)}_r=0$.  

\subsection{Application of cut open null-bordisms to doubling operators}

A different (but related) approach to understanding structure in $\CC$ was suggested in ~\cite{CHL5}, namely considering $\CC$ as a set on which there exists many natural operators. Since one example of such an operator is connected-sum with a fixed knot, this approach can be argued to be more general than focusing on $\CC$ as an abelian group. In fact, it was  suggested in ~\cite{CHL5} that $\CC$ is a \textbf{fractal space}. A fractal space is a metric space which admits systems of natural self-similarities.  Leaving aside (in this paper) the definition of a metric on $\CC$, self-similarities are merely injective functions ~\cite[Definition 3.1]{BGN}. 

The proposed self-similarities are classical \textbf{satellite operations}. Let $R$ be a slice knot and  let $\eta$ be an unknotted curve in $S^3-R$ that has zero linking number with $R$. Since $\eta$ is unknotted its exterior is a solid torus $ST$ and $R$ is a \textbf{pattern knot} of winding number zero in $ST$. Then, given another knot $J$ let $R_\eta(J)$  be the satellite knot of $J$ with pattern $R$ and axis $\eta$.  This correspondence defines what was called in ~\cite{CHL5} a  \textbf{doubling operator},  
\begin{equation}\label{eq:doublingoperator}
R_{\eta}:\mathcal{C}\to\mathcal{C}
\end{equation}
on the set of knot concordance types.  Such operators are rarely homomorphisms, whence arises the focus on $\CC$ as a set. There has been considerable interest in whether such functions are \textbf{weakly injective} (an operator is called weakly injective if $R(J)=R(0)$ implies $J=0$ where $0$ is the class of the trivial knot). For example, it is a famous open problem as to whether or not the Whitehead double operator is weakly injective ~\cite[Problem 1.38]{Kirbyproblemslist} \cite{HedKirk} .

We remark that for genus one knots the approach via doubling operators is equally as general as the approach of studying derivatives. For, if $K$ is a genus one knot with derivative knot $J$,  it was shown in ~\cite[Proposition 1.7]{CFT} that $K$ is concordant to a knot of the form $R_{\eta}(J)$ for some genus one ribbon knot $R$.

 In ~\cite{CHL5} large classes of such operators, called ``robust doubling operators'' were introduced and evidence was presented for their injectivity.  Presently no  doubling operator is known to be  weakly injective. Since it is known that $R_{\eta}(\FF_n)\subseteq \FF_{n+1}$ ~\cite[Lemma 6.4]{CHL4} , a reasonable way to begin to study the weak injectivity of doubling operators is to posit the possible weak injectivity of the maps
\begin{equation}\label{eq:operatoronsolvable}
\frac{\CC}{\mathcal{F}_{n}}\overset{R_\eta}{\longrightarrow}\frac{\CC}{\mathcal{F}_{n+1}},
\end{equation} 
for, if each such map were weakly injective,  then 
$$
R_{\eta}(J)~\text{slice} \Longrightarrow J\in \cap _{n=1}^\infty \FF_n.
$$

Similarly to the case for derivatives, $R_\eta(J)$ being slice ($n$-solvable) implies that $M(J)$ bounds a null-bordism ($n$-null-bordism) so that the same strategy as we outlined in subsection~\ref{subsect:cut open} to study derivatives applies equally well in this setting.

We will point out herein that it is a consequence of the authors' previous work ~\cite{CD1} that there exist robust doubling operators for which the functions of ~(\ref{eq:doublingoperator}) and ~(\ref{eq:operatoronsolvable}) (even $n=0$!) \textit{fail} to be weakly injective.  Indeed there are examples where $R_\eta(J)$ is slice and $J$ has non-vanishing Arf invariant and signature function.  However, a striking consequence of the present work is that for most doubling operators this same failure \textit{does not occur} and we are able to prove a surprisingly strong result in support of weak injectivity.

\begin{thm}\label{thm:doublingopsign} Suppose that $R_{\eta}$ is a doubling operator wherein $\eta$, viewed in the Alexander module of $R$, has annihilator coprime to every non-zero polynomial of the form $(at^p-b)$, $n\in \mathbb{N}, a,b \in\Z$, and for every Lagrangian submodule $Q\subseteq \A(R)$ there is a slice disk $\Delta$ for $R$ such that $\eta\notin Q+\ker(\A(R)\to \A(\Delta))$. If $R_\eta(J)$ is slice, or even $1.5$-solvable, then $J$ has vanishing Levine-Tristram signature function.
\end{thm}

For any knot $K$ with Alexander module $\A(K)$, there is a non-singular Hermitian bilinear form $\Bl_K:\A(K)\times \A(K)\to \frac{\Q(t)}{\Q[t,t^{-1}]}$.  A submodule $P\subseteq \A(K)$ is a called \textbf{Lagrangian} if 
$$P = P^\perp:= \{q\in \A(K) \text{ such that for all }p\in P, ~\Bl_K(p,q)=0\}$$
is its own orthogonal complement.  According to \cite{COT}, if $V$ is a slice disk complement (or even a $1$-solution, see Definition~\ref{defn:n-sol}) for $K$ then $\ker(\A(K)\to \A(V))$ is Lagrangian.  The doubling operator of Figure~\ref{fig:doublingOp} has the desired properties, as it has only one Lagrangian submodule, that submodule corresponds to a slice disk, and $\eta$ is not in that submodule.  See \cite[Section 3]{Davis14}.

\begin{figure}[htbp]
\setlength{\unitlength}{1pt}
\begin{picture}(165,100)
\put(0,0){\includegraphics[height=100pt]{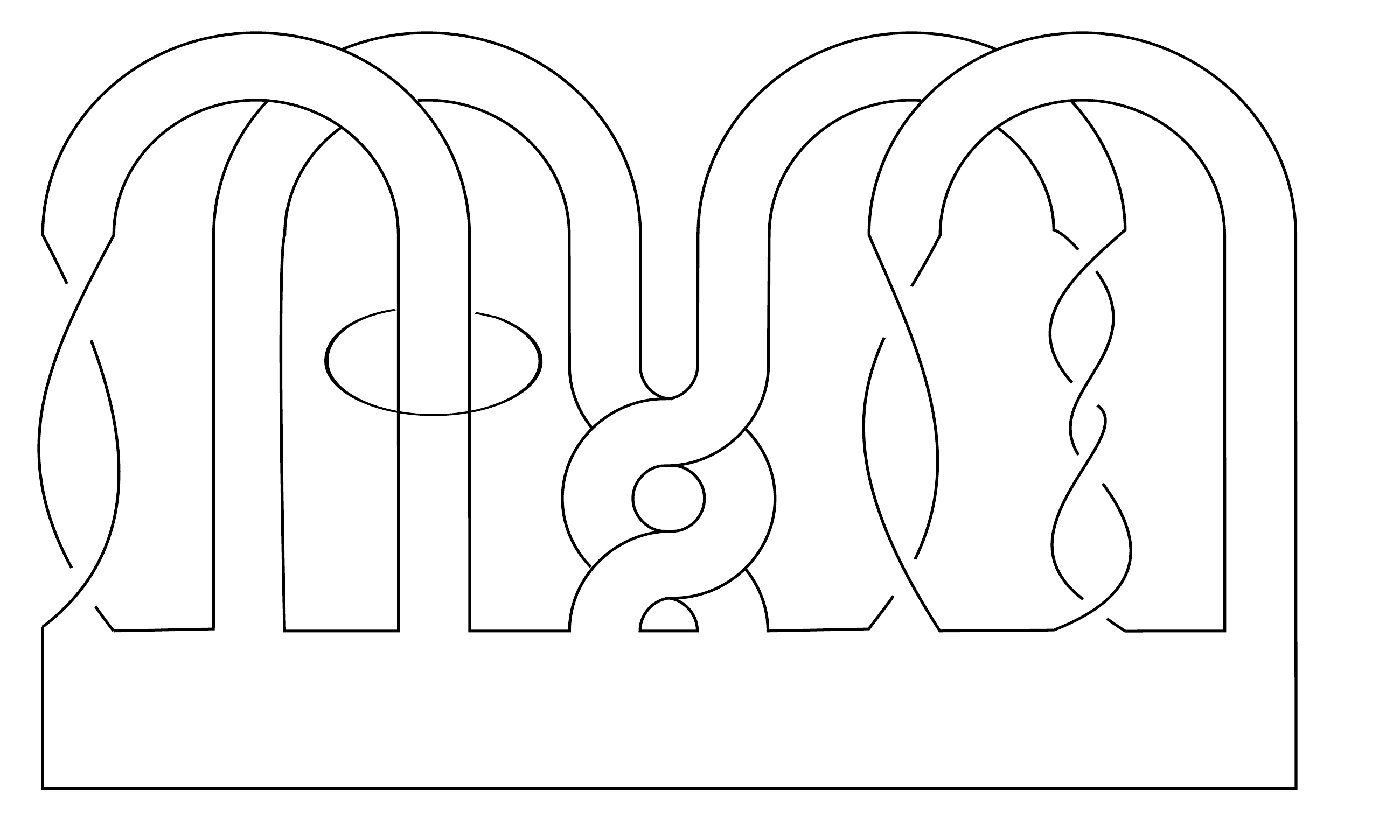}}
\put(60, 65){$\eta$}
\end{picture}
\caption{A doubling operator satisfying the conditions of Theorem~\ref{thm:doublingopsign}.}\label{fig:doublingOp}
\end{figure}

Recall that the signature function determines the class of a knot in $\frac{\CC}{\mathcal{F}^{}_{0.5}}$ modulo torsion.  Thus, Theorem~\ref{thm:doublingopsign} falls only slightly short of the desired goals of proving the weak injectivity of the maps:
$$\frac{\CC}{\mathcal{F}^{}_{0.5}}\overset{R_\eta}{\longrightarrow}\frac{\CC}{\mathcal{F}_{1.5}}
$$

\subsection{Outline of paper} In Section~\ref{sect:bordism} we explore the properties of null-bordisms.  In Section~\ref{sect:cut open} we explicitly describe the procedure of cutting open null-bordisms and study the resulting 4-manifolds.  In Section~\ref{sect:sign} we prove Theorems~\ref{thm:genusonesign} and \ref{thm:doublingopsign} regarding the signature function.  Finally in Section~\ref{sect:suff} we provide a new sufficient condition for a knot to be $1.5$-solvable, proving Theorem~\ref{thm:sufficiency}.  

\subsection*{acknowledgments} Thanks are in order to Shelly Harvey for helpful advise, and to Carolyn Otto, Arunima Ray and Jung Hwan Park for reading this document while it was in preparation.  

\subsection*{Note}  In December of 2014, while this document was still in its preparations, the first author sadly passed away, leaving the second author deeply in his debt.

\section{Null-bordisms and $n$-null-bordisms}\label{sect:bordism}

The main goal of this paper is to address the following questions:

\begin{question}\label{quest:one} If $K$ is a slice knot with slice disk $\Delta$ and $J$ is a derivative of $K$ associated to $\Delta$, then what can we conclude \textit{geometrically} about $J$? 
\end{question}
\begin{question}\label{quest:two}  If $K$ is $n$-solvable via $Y$ ($n\geq 1$) and $J$ is a derivative of $K$ associated to $Y$, then what can we conclude \textit{geometrically} about $J$? 
\end{question}
\begin{question}\label{quest:three}  Suppose $R_\eta$ is a doubling operator wherein $R$ admits a ribbon disk $\Delta$ with associated Lagrangian $P_\Delta$. If $K:=R_\eta(J)$ is a slice knot (or is $n$-solvable) with associated Lagrangian $P_D$ and regarded as an element of the Alexander module $\eta\notin P_D+P_\Delta$, then what can we conclude \textit{geometrically} about $J$? 
\end{question}

For completeness and for ease of reference we recall what it means for a knot to be $n$-solvable.

\begin{defn}[Definition 1.2 of \cite{COT}]\label{defn:n-sol}
Let $K$ be a knot and $n\in \N\cup\{0\}$.  Then $K$ is called $n$-solvable if there exists a 4-manifold $W$ such that $\bdry W = M_K$ is the zero surgery on $K$ and 
\begin{enumerate}
\item The map $H_1(M_K)\to H_1(W)$ is an isomorphism.
\item There exist embedded surfaces $L_1,\dots L_k$, $D_1,\dots D_k$ (called Lagrangians and duals respectively) which are disjoint except that for all $i$, $L_i$ intersects $D_i$ transversely in a single point and such that $\{L_i, D_i\}_{i=1}^k$ forms a basis for $H_2(W)\cong \Z^{2k}$.
\item $\im(\pi_1(D_j)\to \pi_1(W))$ and $\im(\pi_1(L_j)\to \pi_1(W))$ are each contained in $\pi_1(W)^{(n)}_r$.
\end{enumerate}
If additionally, $\im(\pi_1(L_j)\to \pi_1(W))\subseteq \pi_1(W)^{n+1}_r$ then $W$ is an $n.5$-solution and $K$ is $n.5$-solvable.  $\mathcal F_n$ and $\mathcal F_{n.5}$ are the sets of all $n$ and $n.5$-solvable knots respectively.
\end{defn}

Recall that for any group $G$ the \textbf{rational derived series} of $G$, $G^{(n)}_r$ is defined recursively by $G^{(0)}_r=G$, and $G^{(n+1)}_r=\ker(G^{(n)}_r\to H_1(G^{(n)}_r;\Q))$ is the set of all elements of $G_r^{(n)}$ which are of finite order in the abelianization.  This forms the most rapidly descending series for which $G^{(n)}_r/G^{(n+1)}_r$ is torsion-free and abelian.

On our way to a preliminary answer to Questions~\ref{quest:one} and \ref{quest:two} we construct a 4-manifold bounded by $M(J)$ in the case that $K$ is slice and a similar one in the case that $K$ is $n$-solvable.  We begin by discussing the meaning and necessity of the phrases ``associated to $\Delta$'' and ``associated to $Y$''. 

\begin{defn}\label{def:assocLagrangian} Suppose that $Y$ is the exterior of a slice disk $\Delta$ for the knot $K$, or more generally, suppose that $Y$ is a rational $n$-solution 
for $K$ for some $n\geq 1$. Then the kernel of the map of rational Alexander modules
\begin{equation}\label{eqn:alexmaps}
\A(K)\cong \A(M_K)\overset{i_*}{\longrightarrow} \A(Y)
\end{equation}
is the \textbf{Lagrangian submodule $P_Y$ associated to $Y$ (or to $\Delta$)}. 
For any Seifert surface $F$ of genus $g$ for $K$, a derivative $J = \{d_1,\dots d_g\}$ is 
\textbf{associated to $Y$ (or $\Delta$)} if  $\{d_1,...,d_g\}$ generates the rational vector space $P_Y$.
\end{defn}

The existence of a derivative associated to any slice disk, or indeed to any $n$-solution ($n\ge 1$) is given by \cite[Lemma 5.5]{CHL7}.

Since $\frac{\pi_1(Y)^{(1)}_r}{\pi_1(Y)^{(2)}_r}$ is $\Z$-torsion free and $\A(Y)\cong \frac{\pi_1(Y)^{(1)}_r}{\pi_1(Y)^{(2)}_r}\otimes \Q$, it follows that, for each $j$
\begin{equation}\label{eq:dsecondderived}
i_*(d_j)\in \pi_1(Y)^{(2)}_r.
\end{equation}
It also follows that if $\{d_1,...,d_g\}$  is extended to a symplectic basis $\{d_i,b_i\}$ as shown in Figure~\ref{fig:funkytwisty}, and  $\{\alpha_{1},..,\alpha_{g},\beta_1,..,\beta_g\}$ is the basis of $H_1(S^3-\Sigma)$  that is dual to $\{d_1,..,d_{g},b_1,..,b_g\}$ under linking number in $S^3$, then $\{i_*(\alpha_{1}),...,i_*(\alpha_{g})\}$ spans the image of $i_*$ in ~(\ref{eqn:alexmaps} ) ~\cite[Proposition 5.6]{CHL7}. 

\begin{figure}[htbp]
\setlength{\unitlength}{1pt}
\begin{picture}(165,100)
\put(-60,0){\includegraphics{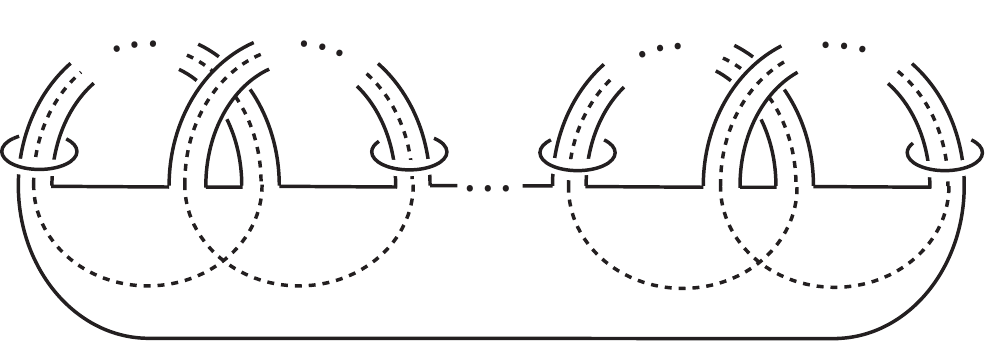}}
\put(-77,53){$\alpha_1$}
\put(-26,21){$d_1$}
\put(22,21){$b_1$}
\put(32,52){$\beta_1$}
\put(122,52){$\alpha_g$}
\put(129,22){$d_g$}
\put(176,22){$b_g$}
\put(228,51){$\beta_g$}
\end{picture}
\caption{Disk-band form for $\Sigma$}\label{fig:funkytwisty}
\end{figure}
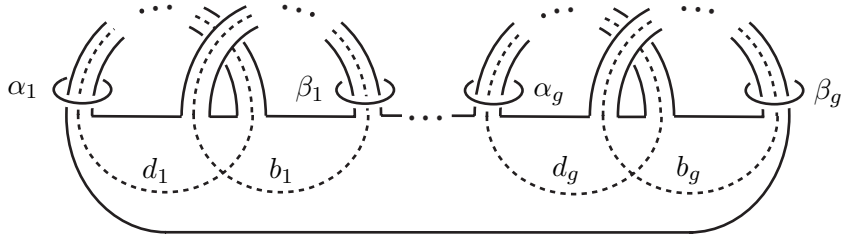
Therefore,
\begin{defn}\label{rem:canonicalmap} If $J$ is associated to $Y$ (or $\Delta$) there is an associated \textbf{canonical map}
$$
f_{J,Y}:\pi_1(M_J)\to \mathcal{A}(Y),
$$
or, in the special case of a slice disk exterior, ($N(\Delta)$ being a tubular neighborhood of $\Delta$)
$$
f_{J,\Delta}:\pi_1(M_J)\to \mathcal{A}(B^4-N(\Delta)),
$$
defined by sending the $j^{th}$ meridian of $J$ to $i_*(\alpha_j)$. Note that this map is non-trivial as long as $\Delta_K(t)\neq 1$ since the image of $i_*$ has $\Q$-rank equal to half the $\Q$-rank of $\A(K)$. 
\end{defn}

We will phrase our answer to Question~\ref{quest:one} (and Questions~\ref{quest:two} and \ref{quest:three}) in terms of the zero-framed surgery manifold $M_J$. The philosophy of Levine (and Kauffman's conjecture) predicts that the answer to Question~\ref{quest:one} is that $J$ is itself a slice knot.  So we should keep in mind what conditions on $M_J$ \textit{would} hold were this conjecture true. Recall that $J$ is slice in a homology $4$-ball if and only if $M_J$ (which is a homology $S^1\times S^2$)  is the boundary of a  homology $S^1\times B^3$. It can easily be seen that this is equivalent to $M_J=\partial W$ where $H_2(W)=0$ and $H_1(M_J)\to H_1(W)$ is injective (excess $H_1$ can be eliminated by surgery). The reader should keep this condition in mind but, since we know that Kauffman's conjecture is false, we expect a weaker relation arises in reality.


The first steps  in answering Questions~\ref{quest:one} and \ref{quest:two} were already taken in ~\cite[Section 8]{CHL7}. There they defined a \textbf{fundamental cobordism}, $E$,  between $M_K$ and $M_J$. This is obtained from $M_K \times [0,1]$ by first adding zero-framed $2$-handles to $M_K \times \{1\}$ along each of the components of $J$. In the new boundary, $K$ bounds a disk, so is isotopic to the trivial knot. Thus, the new boundary is homeomorphic to $M_J\# S^1\times S^2$.  We can add a $3$-handle along the non-seperating $2$-sphere and the resulting $4$-manifold is called $E$. This cobordism captures the key relationship between $K$ and its derivative $J$. As in ~\cite[Lemma 2.5]{CHL3} and ~\cite[Section 8.1]{CHL7},  Mayer-Vietoris sequences and handle arguments establish the following elementary properties of $E$:

\begin{lem}\label{lem:Efacts} The fundamental cobordism $E$ between $M_K$ and $M_J$ satisfies
\begin{enumerate}
\item  $H_1(M_K)\cong H_1(E)\cong \Z$ is generated by $\mu_K$;
\item  $H_2(M_K)\to H_2(E)$ is the zero map;
\item  $H_1(M_J)\to H_1(E)$ is the zero map ;
\item $ H_2(M_J)\to H_2(E)$ is an isomorphism.
\end{enumerate}
\end{lem}

 Now if  $K$ is slice (or $n$-solvable) then $M_K$ bounds a 4-manifold $B^4-N(\Delta)$ (or the $n$-solution $Y$)  which may be glued to $E$ along $M_K$, resulting in a $4$-manifold $V$ whose boundary is $M_J$.  The attributes of this manifold $V$ are described by the proceeding lemma:

\begin{lem}
[Lemma 8.2 and Proposition 8.4 of \cite{CHL7}]

\label{lem:Vfacts} If $K$ admits a  slice disk $\Delta$ and $J$ is a derivative link associated to $\Delta$, then there is a compact, connected, oriented $4$-manifold $V$ with $\bdry V = M(J)$ such that
\begin{enumerate}
\item \label{meridianInV}  The meridian of the $i^{th}$ component of $J$ in $M_{J}$ is isotopic in $V$ to $\alpha_i$ (see Figure~\ref{fig:funkytwisty});
\item  $H_1(V;\mathbb{Z})\cong \Z$ is generated by $\mu_K$;
\item  $H_1(M_J;\mathbb{Z})\to H_1(V;\mathbb{Z})$ is the zero map;
\item\label{Pi1VWhenSlice}  $\pi_1(V)\cong \frac{\pi_1(B^4-N(\Delta))}{<i_*(d_1),...,i_*(d_g)>}$, so $\frac{\pi_1(V)}{\pi_1(V)^{(2)}_r}\cong \frac{\pi_1(B^4-N(\Delta))}{\pi_1(B^4-N(\Delta))_r^{(2)}}$, and
\item  \label{H2VWhenSlice}$H_2(V;\mathbb{Z})/i_*(H_2(\partial V;\mathbb{Z}))=0$.
\end{enumerate}
If $K$ is $n$-solvable with $n$-solution $Y$ ($n\in \N\cup\{0\}$) then (\ref{Pi1VWhenSlice}) and (\ref{H2VWhenSlice}) are replaced with 
\begin{enumerate}
\item[(\ref{Pi1VWhenSlice}')]  $\pi_1(V)\cong \frac{\pi_1(Y)}{<i_*(d_1),...,i_*(d_g)>}$, so $\frac{\pi_1(V)}{\pi_1(V)^{(2)}_r}\cong \frac{\pi_1(Y)}{\pi_1(Y)_r^{(2)}}$,
\item [(\ref{H2VWhenSlice}')]$H_2(V;\mathbb{Z})/i_*(H_2(\partial V;\mathbb{Z}))\cong \Z^{2k}$ has a basis consisting of embedded surfaces $\{L_j, D_j\}$ which are disjoint except that $L_j$ intersects $D_j$ transversely in a single point.  
\item [(\ref{H2VWhenSlice}'a)] $\im(\pi_1(D_j)\to \pi_1(V))$ and $\im(\pi_1(L_j)\to \pi_1(V))$ are each contained in $\pi_1(V)^{(n)}_r$.
\item [(\ref{H2VWhenSlice}'b)] If $Y$ is an $n.5$-solution then additionally
 $\im(\pi_1(L_j)\to \pi_1(V))\subseteq\pi_1(V)^{(n+1)}_r$.
\end{enumerate}
\end{lem}

 A key new feature of $V$ is that the map $H_1(M_J)\to H_1(V)$ is the zero map.  Notice, however,  that as a consequence of result (\ref{meridianInV}), (\ref{Pi1VWhenSlice}), and (\ref{Pi1VWhenSlice}') of Lemma~\ref{lem:Vfacts}, and the observation in Definition~\ref{rem:canonicalmap}, the map $\pi_1(M(J))\to \frac{\pi_1(V)}{\pi_1(V)^{(2)}_r}$ factors non-trivially (provided that $\Delta_K(t)\neq 1$) through the abelianization map $\pi_1(M(J))\to H_1(M(J))$,  and has image contained in $\frac{\pi_1(V)^{(1)}_r}{\pi_1(V)^{(2)}_r}$ which injects into  $\A(V)$.  Motivated by this result we present some definitions:
 
 \begin{defn}[Compare to Definition 8.3 and Section 10 of \cite{CHL7}]\label{defn:NullBordism}
 For a knot $J$, we will call a 4-manifold $V$ with $\bdry V = M(J)$ a \textbf{null-bordism} for $J$ if
 \begin{enumerate}
 \item $H_1(V) = \Z$
 \item $H_1(M(J))\to H_1(V)$ is the zero map.
 \item \label{NullBordismH2}$H_2(V;\mathbb{Z})/i_*(H_2(\partial V;\mathbb{Z}))=0$.
 \end{enumerate}
  For $n\in\N\cup \{0\}$, we call $V$ an \textbf{$n$-null-bordism} if condition (\ref{NullBordismH2}) is replaced by 
  \begin{enumerate}
 \item [(\ref{NullBordismH2}')]  $H_2(V;\mathbb{Z})/H_2(\partial V;\mathbb{Z})\cong \Z^{2k}$ has a basis consisting of embedded surfaces $\{L_j, D_j\}$ which are disjoint except that $L_j$ intersects $D_j$ transversely in a single point.  
\item 
[(\ref{NullBordismH2}'a)]
\label{nBordism} $\im(\pi_1(D_j)\to \pi_1(V))$ and $\im(\pi_1(L_j)\to \pi_1(V))$ are each contained in $\pi_1(V)^{(n)}_r$.
 \end{enumerate}
 $V$ is an $n.5$-null-bordism if additionally
 \begin{enumerate}
\item [(\ref{NullBordismH2}'b)]\label{n.5Bordism} $\im(\pi_1(L_j)\to \pi_1(V))\subseteq\pi_1(V)^{(k+1)}_r$.
 \end{enumerate}
 We call $V$ \textbf{non-degenerate} if the map $H_1(M(J))\to \A(V)$ is not the zero homomorphism.
 \end{defn}

In \cite{CHL7} the notation ``height 1'' is used where we say ``non-degenerate.''  These 4-manifolds give already a partial answer to Questions~\ref{quest:one} and \ref{quest:two}:  

\begin{prop}\label{prop:NullBordismDeriv}If $K$ is slice and $J$ is a derivative associated with a slice disk.  Then $M_J$ bounds a null-bordism, $V$.  Moreover $K$ has Alexander polynomial $\Delta_K(t) = \delta(t)\delta(t^{-1})$ for some polynomial $\delta(t)$ such that the image of $H_1(V;\Z)\to \A(V)$ is contained in the $\delta(t)$-torsion submodule.    If $K$ is instead $n$-solvable (or $n.5$-solvable) with $n\in \N$ then $V$ is an $n$-null-bordism (or $n.5$-null-bordism).  If $\Delta_K\neq 1$, then $V$ is non-degenerate.
\end{prop}
Before we refine this answer (in the next section) we perform an analogous construction in the situation of Question~\ref{quest:three}.

The first step is to see that there is a \textit{fundamental infection cobordism}, $E$, between $M_{R_\eta(J)}$ and $M_J\sqcup M_R$.  This cobordism is obtained by starting with ($M_R\times [0,1])\sqcup (-M_J\times [0,1])$ and identifying a tubular neighborhood of $\eta$ in $M_R\times \{1\}$ with a tubular neighborhood of $J$ in $M_J\times \{0\}$ in such a way that a longitude of $\eta$ is identified to a meridian of $J$ and a meridian of $\eta$ is identified to the reverse of a longitude of $J$. Then $\partial E=M_{R_\eta(J)}\sqcup -M_R\sqcup -M_J$.  Its important properties are summarized as:
\begin{enumerate}
\item $H_1(E)\cong \Z$ is generated by the meridian of $R$, which is isotopic in $E$ to the meridian of $R_\eta(J)$.
\item The maps $\A(R)\to \A(V)\ot \A(R_\eta(J))$ are isomorphisms and the composition $\A(R)\to \A(R_\eta(J))$ preserves the Blanchfield form.  
\item In $V$, the meridian of $J$ is isotopic to $\eta$, which is null-homologous, so $H_1(M_J)\to H_1(V)$ is the zero homomorphism.
\end{enumerate}

Since the maps $\A(R)\to \A(V)\ot \A(R_\eta(J))$ induced by inclusion into $V$ are isomorphisms, we may consistently think of a submodule $P\subseteq \A(R_\eta(J))$ as a submodule of $\A(R)$.  

Let $\Delta$ be a slice disk for $R$.  Let $Y$ be a slice disk complement (or $n$-solution) for $R_\eta(J)$.   To get a null-bordism (or $n$-null-bordism) bounded by $M_J$ glue to $E$ copies of $E(\Delta)$ and $Y$.  Call this $4$-manifold $V$.   In the case the $R_\eta(J)$ is slice, this amounts to gluing homology $S^1\times D^3$'s to homology $S^1\times S^2$'s.  This operation has no effect on $H_1$ and kills the second homology generated by the boundary components used for the gluing.  Thus, if $R_\eta(J)$ is slice then $V$ is a null-bordism.  In the case $Y$ is an $n$-solution for $R_\eta(J)$, it follows that $H_2(V;\mathbb{Z})/i_*(H_2(\partial V;\mathbb{Z}))\cong H_2(Y, \bdry Y)$ so that we can conclude that $V$ is an $n$-null-bordism.

Finally, observe that by a Mayer-Veitoris argument, 
$$\ker(\A(R)\to \A(V)) = \ker(\A(R)\to \A(E(\Delta)))+\ker(\A(R)\to \A(E\cup Y)) = P_\Delta+P_Y.$$
If $\eta\notin P_\Delta+P_Y$ then $V$ is non-degenerate.  To summarize:

\begin{prop}\label{Prop:NullBordismInfection}
Suppose that $R_\eta$ is a doubling operator, that $R$ is slice with corresponding Lagrangian $P_\Delta\subseteq \A(R)$ and that $R_\eta(J)$ is slice (or $n$-solvable or $n.5$-solvable) with slice disk complement (or $n$-solution or $n.5$-solution) $Y$.  Then $J$ bounds a null-bordism (or $n$-null-bordism or $n.5$-null-bordism) $V$.  In $\A(V)$ the meridian of $J$ is annihilated by $\Delta_R(t)$.  Moreover the map $H_1(M(J))\to \frac{\pi_1(V)}{\pi_1(V)^{(2)}_r}$ factors through $H_1(M(J))\underset{\mu_J\to \eta}{\longrightarrow}\dfrac{\A(R)}{P_\Delta+P_Y}\into\frac{\pi_1(V)}{\pi_1(V)^{(2)}_r}$.  Thus, if $\eta\notin P_\Delta+P_Y$ then $V$ is non-degenerate.
\end{prop}

Thus, our answers to the Questions \ref{quest:one}, \ref{quest:two} and \ref{quest:three} are that $M(J)$ bounds a null-bordism or $n$-null-bordism as appropriate.  With some assumptions about the Alexander module, we can further conclude that this null-bordism is non-degenerate.  Before we move on to the next section in which we cut null-bordisms open we explore their higher order properties.  

\subsection{Higher order properties of $n$-null-bordisms}

In general the groups employed in this paper are PTFA (Poly-Torsion-Free-Abelian) groups.  A group $\Gamma$ is PTFA if for some $k\in \N$, $\Gamma^{(k)}_r=\{1\}$ is the trivial group.  As we mentioned in the introduction if $\Gamma$ is PTFA then $\Q[\Gamma]$ is an Ore domain and so embeds in a skew-field of fractions $\K(\Gamma)$. See\cite[Chapter 2]{ste}) and \cite[Proposition 32]{C}.)

If $V$ is an $n$-null-bordism and $\phi:\pi_1(V)\to \Gamma$ is a homomorphism to a PTFA group, $\Gamma$, with $\Gamma^{(n)}=1$ then since $\pi_1(L_i)$ and $\pi_1(D_i)$ each sit in $\pi_1(V)^{(n)}$, $L_i$ and $D_i$ lift to the $\Gamma$-cover of $V$, and so represent elements of $H_2(W;\Q[\Gamma])$.  Additionally,  in \cite[Lemma 5.10]{CHL3} the proposition below is proven, where $\beta_*(X)$ refers to the $\Q$-rank of $H_*(X;\Q)$ and $\beta_*^\Gamma(X)$ is the $\K(\Gamma)$-rank of $H_*(X;\K(\Gamma))$.  Importantly, according to  \cite[Proposition 3.10]{C} the Betti number hypothesis is satisfied whenever $\bdry V$ is zero surgery on a knot and $\phi$ is non-trivial on $\pi_1(\bdry V)$.

\begin{prop}[Lemma 5.10 of \cite{CHL3}]\label{highOrderSoln}
Let $n\in \N$.  Suppose that $V$ is an $n$-null-bordism with Lagrangians $\{L_1,\dots, L_k\}$ and duals $\{D_1,\dots, D_k\}$ and $\phi:\pi_1(V)\to \Gamma$ is a homomorphism to a PTFA group with $\Gamma^{(n)}=1$.  Suppose also that either
$\beta_1^\Gamma(\bdry V) = \beta_1(\bdry V)-1$ or the homomorphism $\pi_1(\bdry V)\to \Gamma$ is trivial.  
Then the collection of lifts $\{\widetilde L_1, \dots\widetilde L_k, \widetilde D_1, \dots\widetilde D_k\}$ forms a basis for $\dfrac{H_2(V;\K(\Gamma))}{H_2(\bdry V;\K(\Gamma))}$
\end{prop}

In the case that $V$ is instead an $n.5$-null-bordism and $\Gamma^{(n+1)}=1$ we get a weaker result.

\begin{prop}\label{highOrderSoln.5}
Let $n\in \N\cup\{0\}$.  Suppose that $V$ is an $n.5$-null-bordism with Lagrangians $\{L_1,\dots, L_k\}$ and $\phi:\pi_1(V)\to \Gamma$ is a homomorphism to a PTFA group with $\Gamma^{(n)}=1$.  Suppose also that either 
$\beta_1^\Gamma(\bdry V) = \beta_1(\bdry V)-1$  or $\pi_1(\bdry V)\to \Gamma$ is trivial.
Then the lifts $\widetilde L_1, \dots\widetilde L_k$  are $\K(\Gamma)$-linearly independent in $\dfrac{H_2(V;\K(\Gamma))}{H_2(\bdry V;\K(\Gamma))}$ and $\rank_{\K(\Gamma)}\left(\dfrac{H_2(V;\K(\Gamma))}{H_2(\bdry V;\K(\Gamma))}\right)=2k$.
\end{prop}
\begin{proof}

In the case that $\pi_1(V)\to \Gamma$ is the trivial homomorphism, then the proof is immediate.  

Our proof is based on the proofs of Theorem 4.2 and Proposition 4.3 of \cite{COT}.  We begin by showing that 
\begin{equation}\label{twisted = untwisted}\rank_{\K(\Gamma)}\left(\dfrac{H_2(V;\K(\Gamma))}{H_2(\bdry V;\K(\Gamma))}\right)
=
\rank_{\Q}\left(\dfrac{H_2(V;\Q)}{H_2(\bdry V;\Q)}\right).
\end{equation}
The latter is $2k$, as $\{L_i, D_i\}$ is a  basis.  

In order to prove equation (\ref{twisted = untwisted}) consider the exact sequence with coefficients in $\Q$ or in $\K(\Gamma)$
\begin{equation}\label{GenericSeq}
0\to H_4(V;\bdry V)\to H_3(\bdry V)\to H_3(V)\to H_3(V,\bdry V)\to H_2(\bdry V)\to H_2(V)\to \dfrac{H_2(V)}{H_2(\bdry V)}\to 0
\end{equation}
Over $\Q$ we have that each of $H_4(V;\bdry V)$, $H_3(\bdry V)$, $H_3(V)$, $H_3(V,\bdry V)$ is rank $1$.  It follows that $H_2(\bdry V;\Q)\to H_2(V;\Q)$ is injective and
\begin{equation}\label{UntwistedRank}
\rank_\Q\left(\dfrac{H_2(V;\Q)}{H_2(\bdry V;\Q)}\right) = \beta_2(V)-\beta_2(\bdry V) = \beta_2(V)-\beta_1(\bdry V)
\end{equation}
 The second equality follows by Poincar\'e duality.

On the other hand $H_3(V,\bdry V;\K(\Gamma))\cong H_1(V;\K(\Gamma))=0$, by the Poincar\'e duality isomorphism and \cite[Proposition 3.10]{C}.  Thus, $H_2(\bdry V;\K(\Gamma))\to H_2(V;\K(\Gamma))$ is injective and 
\begin{equation}\label{TwistedRank} \rank_{\K(\Gamma)}\left(\dfrac{H_2(V;\K(\Gamma))}{H_2(\bdry V;\K(\Gamma))}\right) = \beta^\Gamma_2(V)-\beta^\Gamma_2(\bdry V) = \beta^\Gamma_2(V)-\beta^\Gamma_1(\bdry V) \end{equation}
By assumption $\beta^\Gamma_1(\bdry V) = \beta_1(\bdry V)-1$.   It remains to compare $\beta_2(V)$ with $\beta_2^\Gamma(V)$.  By Propositions 3.7 and 3.10 of \cite{C}, $\beta_0^\Gamma(V)= \beta_1^\Gamma(V) = 0$.  By Poincar\'e duality, $\beta_4^\Gamma(V)=0$.  Taking advantage of the exact sequence,  $0= H_1(V;\K(\Gamma))\to H_1(V,\bdry V;\K(\Gamma))\to H_0(\bdry V;\K(\Gamma))= 0$,  we see $H_3(V;\K(\Gamma))\cong H_1(V,\bdry V;\K(\Gamma)) = 0$, and $\beta^\Gamma_3(V)=0$.  Since $\beta_k^\Gamma(V)=0$ for all $k\neq 2$ and the alternating sum of ranks of twisted homology gives the Euler characteristic, we conclude $\chi(V)=\beta^\Gamma_2(V)$.  

On the other hand $\beta_0(V)=\beta_1(V)=\beta_3(V)=1$ and $\beta_4(V)=0$ so that $\chi(V)=\beta_2(V)-1$, and in particular $\beta_2^\Gamma(V) = \beta_2(V)-1$We new see that the right hand sides of equations (\ref{UntwistedRank}) and (\ref{TwistedRank}) coincide.  Equation~(\ref{twisted = untwisted}) follows.  

Next we verify equation~(\ref{twisted = untwisted}) in the case that $\pi_1(\bdry V)\to \Gamma$  is the trivial homomorphism and $\pi_1( V)\to \Gamma$ is nontrivial.  By the triviality assumption, $\beta_*(\bdry V) = \beta_*^\Gamma(\bdry V)$.   Using the exactness of 
$$
H_1(V;\K(\Gamma))\to H_1(V, \bdry V;\K(\Gamma))\to H_0(\bdry V;\K(\Gamma))\to H_0(V;\K(\Gamma))
$$
and the fact that $H_1(V;\K(\Gamma))$ and $H_0(V;\K(\Gamma))$ each vanish while $H_0(\bdry V;\K(\Gamma))$ is rank 1, it follows that $H_3(V;\K(\Gamma))\cong H_1(V,\bdry V;\K(\Gamma))$ is rank $1$.  Similarly to before, $\beta_0^\Gamma(V) = \beta_1^\Gamma(V)=\beta_4^\Gamma(V)=0$.  Thus, $\beta_2^\Gamma(V) = \chi(V)+1$, which as previously stated, is exactly $\beta_2(V)$.  

Now we prove the linear independence claim.    Let $X$ be the disconnected 2-complex given by the union of the Lagrangians $L_1\sqcup\dots\sqcup L_k$.  Since $\pi_1(L_i)\to \Gamma$ is trivial for each $i$, 
$$H_2(X \sqcup \bdry V, \bdry V;\K(\Gamma))\cong  \oplus_{i=1}^kH_2(L_i;\Q)\otimes_\Q \K(\Gamma)\cong \K(\Gamma)^k$$ is freely generated by lifts of fundamental classes of the Lagrangians.  In order to conclude that the Lagrangians are linearly independent in $H_2(V;\K(\Gamma))/H_2(\bdry V;\K(\Gamma))$ we need to show $H_2(X \sqcup \bdry V, \bdry V;\K(\Gamma))\to H_2(V, \bdry V;\K(\Gamma))$ is injective.  

Let $\ell$ be a simple closed curve which generates $H_1(V)$ and $N\cong S^1\times D^3$ be a tubular neighborhood of $\ell$.  We will also assume that $N$ is disjoint from $X$ and from the duals $D_1,\dots, D_k$.  Since $N$ is a homotopy 1-complex and the map $\pi_1(N)\to\Gamma$ is non-trivial, Propositions 3.7 and 3.10 of \cite{C} conclude that $H_*(N;\K(\Gamma))=0$.    Consider the exact sequence of the triple $(V,V-N, \bdry V)$.
$$
H_{p+1}(V, V-N;\K(\Gamma))\to H_p(V-N,\bdry V;\K(\Gamma))\to H_p(V,\bdry V;\K(\Gamma))\to H_p(V, V-N;\K(\Gamma)).
$$ 
By Excision and Poincar\'e duality $H_*(V, V-N;\K(\Gamma))\cong H_*(N, \bdry N;\K(\Gamma))\cong H_{4-*}(N;\K(\Gamma))=0$.   Thus, the inclusion induced map $H_p(V-N,\bdry V;\K(\Gamma))\to H_p(V,\bdry V;\K(\Gamma))$ is an isomorphism.  Instead of trying to show that $H_2(X \sqcup \bdry V, \bdry V;\K(\Gamma))\to H_2(V, \bdry V;\K(\Gamma))$ is injective it will suffice to show that $H_2(X \sqcup \bdry V, \bdry V;\K(\Gamma))\to H_2(V-N, \bdry V;\K(\Gamma))$ is injective.

  According to the exact sequence of the triple $(V-N, X\sqcup \bdry V, \bdry V)$, 
$$H_3(V-N, X\sqcup \bdry V;\K(\Gamma))\to H_2(X \sqcup \bdry V, \bdry V;\K(\Gamma))\to H_2(V-N, \bdry V;\K(\Gamma)),$$
it will furthermore suffice to show that $H_3(V-N, X\sqcup \bdry V;\K(\Gamma))=0$.  Recall the following fact due to Strebel.
\begin{prop}[Page 305 of \cite{Str}. See also Proposition 3.4 of \cite{C}]\label{prop:Strebel}
If $\Gamma$ is a PTFA group, then any map between projective $\Q[\Gamma]$-modules whose image under the functor $-\otimes_{\Q[\Gamma]}\Q$ is injective is itself injective.  
\end{prop}
Since $(V-N, X\sqcup \bdry V)$ is a homotopy 3-complex, we can (up to chain homotopy equivalence) realize the chain complex $C_*(V-N, X\sqcup \bdry V;\Q[\Gamma])$ as
$
C_3^\Gamma\overset{\bdry_3^\Gamma}{\to} C_2^\Gamma\to C_1^\Gamma\to C_0^\Gamma
$.  
Applying the functor $-\otimes_{\Q[\Gamma]}\Q$ we get the untwisted complex $C_*(V-N, X\sqcup \bdry V;\Q)$:
$
C_3\overset{\bdry_3}{\to} C_2\to C_1\to C_0
.$
Thus, if we can show that $H_3(V-N, X\sqcup \bdry V;\Q)=0$, then $\bdry_3$ is injective.  Finally by Proposition~\ref{prop:Strebel}, $\bdry_3^\Gamma$ is injective and $H_3(V-N, X\sqcup \bdry V;\Q[\Gamma])=0$.  Since $\K(\Gamma)$ is a flat $\Q[\Gamma]$-module, it will follow that $H_3(V-N, X\sqcup \bdry V;\K(\Gamma))=0$, and we will be done.

All that remains is to demonstrate that $H_3(V-N, X\sqcup \bdry V;\Q)=0$.   Consider the untwisted exact sequence of the triple $(V-N, X\sqcup \bdry V, \bdry V)$
\begin{equation}\label{Mayer-VeitorisMinusN}H_3(V-N,\bdry V;\Q)\to H_3(V-N, X\sqcup \bdry V;\Q)\to H_2(X\sqcup \bdry V,\bdry V;\Q)\to H_2(V-N,\bdry V;\Q)
\end{equation}
By Poincar\'e-Lefshetz duality, $H_3(V-N,\bdry V;\Q)\cong H_1(V-N,\bdry N;\Q)$, which fits into the exact sequence
$$
H_1(\bdry N;\Q)\to H_1(V-N;\Q)\to H_1(V-N,\bdry N;\Q)\to 0
$$
Since $V$ is gotten by adding to $V-N$ a 3-handle and a 4-handle, $H_1(V-N;\Q)\cong H_1(V;\Q)\cong \Q$.  By design, the generator of $H_1(V)$ is carried by $H_1(\bdry N;\Q)$.  Thus, $H_1(\bdry N;\Q)\to H_1(V-N;\Q)$ is onto and we conclude $H_3(V-N,\bdry V;\Q)\cong H_1(V-N,\bdry N;\Q)=0$.  

Next we show that the map $H_2(X\sqcup \bdry V,\bdry V;\Q)\to H_2(V-N,\bdry V;\Q)$ of (\ref{Mayer-VeitorisMinusN}) is injective.  Notice, $H_2(X\sqcup \bdry V,\bdry V;\Q)\cong H_2(X;\Q)$ has basis given by the fundamental classes of the Lagrangians $\{L_i\}$.  If some linear combination $\Sum_{i=1}^k a_i L_i$ were zero in $H_2(V-N,\bdry V;\Q)$ then for each dual $D_j$, it would follows that 
$$0 = D_j\cdot \Sum_{i=1}^k a_i L_i = a_j$$
where $\cdot$ indicates the intersection form which counts intersection points with multiplicity.  Since each $a_j=0$ we conclude that $H_2(X\sqcup \bdry V,\bdry V;\Q)\to H_2(V-N,\bdry V;\Q)$ is injective.



By the exactness of (\ref{Mayer-VeitorisMinusN}) we see $H_3(V-N, X\sqcup \bdry V;\Q)=0$.  This completes the proof. 
\end{proof}

\section{Cutting open null-bordisms.}\label{sect:cut open}

If $J$ is a derivative of $K$ which is associated to a slice disk for $K$ and $\Delta_K(t)\neq 1$, then by Proposition~\ref{prop:NullBordismDeriv} $M(J)$ bounds a non-degenerate null-bordism.  In this section we start with the hypothesis that $M(J)$ bounds a null-bordism $V$, (or an $n$-null-bordism) and alter it.    

Proceeding with our discussion,  note that since $H_1(V)\cong \Z$, $V$ admits a unique connected infinite cyclic cover $\widetilde{V}$. We describe a manifold, $W$, that is  a connected fundamental domain for the action of the deck translation group on $\widetilde{V}$. Let $g:V\to S^1$ be a smooth map that induces an isomorphism on $H_1$. Let $\Sigma$ be the preimage of a regular value. Since $H_1(M_J;\mathbb{Z})\to H_1(V;\mathbb{Z})$ is the zero map, we can arrange that $\Sigma$ lies in the interior of $V$ and we can also arrange that $\Sigma$ is connected. Of course $\Sigma$ has a bicollar and is orientable. This $\Sigma$ is not unique. If $V$ is an $n$-null-bordism ($n\ge 1$),  we will have need to choose $\Sigma$ to be disjoint from the Lagrangians $L_i$ and the duals $D_i$.  The following lemma reveals that we can do so.  

\begin{lem}\label{Lem:Sigma}
Suppose that $V$ is a $1$-null-bordism.  Let $\{L_i, D_i\}$ be the basis for $H_2(V)/H_2(\bdry V)$ given by Definition~\ref{defn:NullBordism}. There exists an embedded 3-manifold $\Sigma$ in $V$ which intersects a generator of $H_1(V)$ transversely in a single point and which is disjoint from every $L_i$ and $D_i$.
\end{lem}
\begin{proof}[Proof of Lemma \ref{Lem:Sigma}]
Pick $\Sigma$ as in the discussion prior to the Lemma.  Since $\Sigma$ does not separate, there is a simple closed curve $\ell$ in $V$ which intersects $\Sigma$ transversely in a single point.  This $\ell$ generates $H_1(V)$.  Since $V$ is a 4-manifold, $\ell$ can be assumed to be disjoint from the surfaces $L_i$ and $D_i$.    The proof will proceed by modifying $\Sigma$.

Up to isotopy, assume that $\Sigma$ is transverse to the surfaces $L_i$ and $D_i$ for all $i$.  Then $L_i\cap \Sigma$ consists of simple curves $\alpha_1,\dots \alpha_n$.  If some $\alpha_k$ failed to separate $L_i$ then its dual curve $\beta$ in $L_i$ would intersect $\Sigma$ transversely in a single point and $\beta$ would be non-zero in $H_1(V)$, in contradiction to the property (\ref{NullBordismH2}'a) of Definition \ref{defn:NullBordism} since $V$ is a $1$-null-bordism.  Thus, $\{\alpha_k\}_{k=1}^m$ separates $L_i$ into $m+1$ components:  $L_i^1, \dots, L_i^{m+1}$.  $D_i$ intersects one of these (say $L_i^1$) transversely in a single point, and is disjoint form the rest.  Let $\alpha_1,\dots,\alpha_j$ be the boundary of $L_i^{m+1}$.  We now modify $\Sigma$ by cutting out tubular neighborhoods of $\alpha_1,\dots \alpha_j$ and gluing in the boundary of a tubular neighborhood of $L_i^{m+1}$.  One can immediately check that $\Sigma\cap L_i$ has fewer components and that otherwise the intersection of $\Sigma$ with the Lagrangians, the duals, and $\ell$ is unchanged.  By inducting on the number of components of $L_i\cap \Sigma$  we  arrange that $\Sigma\cap L_i = \emptyset$.  Performing this procedure for each $L_i$, $\Sigma\cap L_i = \emptyset$ for all $i$.  The same argument can be used to arrange that $\Sigma\cap D_i=\emptyset$.
\end{proof}

We can now define the 4-manifold which we will use to study $J$.

\begin{defn}\label{defn: cut open}
Let $V$ be a null-bordism for $J$ and $\Sigma$ be an embedded 3-manifold dual to the generator of $H_1(V)\cong \Z$.  Then $\Sigma$ has an open product neighborhood $N$.  The submanifold $X(V,\Sigma) := V-N$ is \textbf{the result of cutting $V$ open along $\Sigma$}.  We will often leave $\Sigma$ out of the notation, saying $X(V)$.  Let $n\in \N$.  If $V$ is instead an $n$-null-bordism (or $n.5$-null-bordism) then we require that $\Sigma$ satisfy the conclusions of Lemma~\ref{Lem:Sigma}.
\end{defn}

The remainder of this section is devoted to the study of the intersection form on the second homology of cut open null-bordisms.  

\begin{prop}\label{prop:IsectnForm}
Let $n\in \N$ and Let $V$ be an $n$-null-bordism.  Let $\{L_1,\dots, L_k, D_1,\dots, D_k\}$ be the Lagrangians and duals of Definition~\ref{defn:NullBordism}  (If $V$ is a null-bordism then $\{L_i, D_i\} = \emptyset$.)  Let $W=X(V)$ be the resulting cut open $n$-null-bordism.  Then $\{L_i, D_i\}$ forms a basis for $\frac{H_2(W)}{H_2(\bdry W)}\cong \Z^{2k}$.  
\end{prop}
\begin{proof}

Suppose that some linear combination $F = \Sum a_iL_i + \Sum b_i D_i$ is zero in $H_2(W)/ H_2(\bdry W)$.  Then it would follow that the intersection form $F\cdot D_j=0$ for all $j$, so that $a_j=0$.  Similarly $b_j=0$ and $\{L_i, D_i\}$ is linearly independent.  

Next we prove that $\{L_i, D_i\}$ spans.  Consider any $Y\in H_2(W)$.  It suffices to show that there is a linear combination $C = \Sum a_i L_i+\Sum b_iD_i$ and an element $Z\in H_2(\bdry W)$ such that $Y = C+Z$ in $H_2(W)$.   Consider the following Mayer-Veitoris Exact sequence:
\begin{equation}\label{MayerVeitorisW}
H_2(\Sigma^+)\oplus H_2(\Sigma^-)\overset{\phi}{\to} H_2(W)\oplus H_2(N)\overset{\psi}{\to} H_2(V)
\end{equation}
Where $N\cong \Sigma\times [-1,+1]$ is a closed product neighborhood of $\Sigma$ and $\Sigma^+$ and $\Sigma^-$ are the boundary components of $N$.  Since $\{L_i, D_i\}$ spans $H_2(V)/H_2(\bdry V)$ by Definition~\ref{defn:NullBordism} (\ref{NullBordismH2}'), there is a linear combination $C = \Sum a_i L_i+\Sum b_i D_i$ and a class $Z_0\in H_2(\bdry V)$ such that $Y-C-Z_0\in \ker(\psi) = \im(\phi)$.  Thus, there is a class $Z_1\in H_2(\Sigma^+)\oplus H_2( \Sigma^-)$ such that $Y = C+Z_0+Z_1$ in $H_2(W)$.  Since $Z_0$ and $Z_1$ are carried by $\bdry W = \bdry V\sqcup \Sigma^+\sqcup \Sigma^-$ this completes the proof.  
\end{proof}

%

The next step will be to prove twisted homology analogues of Proposition~\ref{prop:IsectnForm}.  Let $n\in \N$ and $V$ be an $n$-null-bordism for the link $J$.  Let $f:\pi_1(V)\to \Gamma$ be a homomorphism to a  PTFA group $\Gamma$ with $\Gamma^{(n)}_r=0$.  This induces a homomorphism on $\pi_1(W)$ by the composition:
$$\pi_1(W)\to \pi_1(V)\overset{f}\to \Gamma.$$  Since $V$ is an $n$-null-bordism, the surfaces $L_i$ and $D_i$ all lift to the $\Gamma$-cover.  Denote their lifts by $\widetilde L_i$ and $\widetilde D_i$.  The following theorem reveals that $\{\widetilde L_i, \widetilde D_i\}$ is a basis for $\frac{H_2(W;\K(\Gamma))}{H_2(\p W;\K(\Gamma))}$.

\begin{thm}\label{thm:nIsectnForm} Let $n\in \N$.  Suppose that $V$ is an $n$-null-bordism and $\psi:\pi_1(V)\to \Gamma$ is a homomorphism where $\Gamma$ is PTFA and $\Gamma^{(n)}_r=0$. Suppose $W = X(V)$ is the resulting cut open null-bordism.  Suppose additionally that either $\psi:\pi_1(\bdry  V)\to \Gamma$ is trivial, or that $\beta_1^\Gamma(\bdry V)= \beta_1(\bdry V)-1$. 
Then $\{\widetilde L_i, \widetilde D_i\}$ forms a basis for
$
\frac{H_2(W;\K(\Gamma))}{H_2(\p W;\K(\Gamma))}.
$
\end{thm}
\begin{proof}


By choosing lifts appropriately, we have that $\widetilde L_i$ and $\widetilde D_i$ intersect transversely in a single point and otherwise are disjoint, and that for any deck transformation $\gamma\in \Gamma$ other than the identity and any $i,j$, $\gamma\left[\widetilde L_i\right]\cap \widetilde L_j = \gamma\left[\widetilde L_i\right]\cap  \widetilde D_j = \gamma\left[\widetilde D_i\right]\cap  \widetilde L_j = \gamma\left[\widetilde D_i\right]\cap  \widetilde D_j = \emptyset$.  As in the proof of Proposition~\ref{prop:IsectnForm} the intersection form on $\frac{H_2(W;\K(\Gamma))}{H_2(\bdry W;\K(\Gamma))}$ reveals that $\{\widetilde{L_1}, \widetilde{D_1},\dots \widetilde{L_k}, \widetilde{D_k}\}$ is a $\K(\Gamma)$-linearly independent set.  If we can now show that $\rank_{\K(\Gamma)}\left(\frac{H_2(W;\K(\Gamma))}{H_2(\bdry W;\K(\Gamma))}\right)\le 2k$ then we are done.

By Proposition~\ref{highOrderSoln} (or \cite[Lemma 5.10]{CHL3}) $\rank_{\K(\Gamma)}\left(\frac{H_2(V;\K(\Gamma))}{H_2(\bdry V;\K(\Gamma))}\right) = 2k$.  The remainder of the proof consists of showing that $\rank_{\K(\Gamma)}\left(\frac{H_2(W;\K(\Gamma))}{H_2(\bdry W;\K(\Gamma))}\right)\le \rank_{\K(\Gamma)}\left(\frac{H_2(V;\K(\Gamma))}{H_2(\bdry V;\K(\Gamma))}\right)$.  Consider the following exact sequence with coefficients in $\K(\Gamma)$ where $Z=W$ or $Z=V$.  ($H_4(Z;\K(\Gamma))=0$ since $\bdry Z\neq \emptyset$.)
\begin{align}
0 \to H_4(Z,\bdry Z)\to H_3(\bdry Z) \to H_3(Z)\to H_3(Z,\bdry Z)\to H_2(\bdry Z)\to H_2(Z)\to \frac{H_2(Z)}{H_2(\bdry Z)}\to 0\nonumber
\end{align}

Let $b(Z) = \rank_{\K(\Gamma)}\left( \frac{H_2(Z;\K(\Gamma))}{H_2(\bdry Z;\K(\Gamma))}\right)$.  Using Poincar\'e duality, the above exact sequence reveals that
\begin{align}
b(V) &= \beta^\Gamma_2(V) - \beta^\Gamma_2(\bdry V) + \beta^\Gamma_1(V)-\beta^\Gamma_3(V)+\beta^\Gamma_3(\bdry V)-\beta^\Gamma_0(V) \nonumber
\\
b(W) &= \beta^\Gamma_2(W) - \beta^\Gamma_2(\bdry W) + \beta^\Gamma_1(W)-\beta^\Gamma_3(W)+\beta^\Gamma_3(\bdry W)-\beta^\Gamma_0(W). \nonumber
\end{align}
  Recall now that $\bdry V=M(J)$ and $\bdry W$ consists of one copy of $M(J)$ and two of $\Sigma$, so that $\beta^\Gamma_*(\bdry W) = \beta^\Gamma_*(\bdry V)+2\beta^\Gamma_*(\Sigma)$.  Taking advantage of this, 
\begin{align}\label{eqn:rankDifference}
b(V)-b(W) =& (\beta^\Gamma_2(V)-\beta^\Gamma_2(W)) + (\beta^\Gamma_1(V)-\beta^\Gamma_1(W))-(\beta^\Gamma_3(V)-\beta^\Gamma_3(W))
\\&-(\beta^\Gamma_0(V)-\beta^\Gamma_0(W)) +2( \beta^\Gamma_2(\Sigma)-\beta^\Gamma_3(\Sigma))\nonumber
\end{align}

In order to compute $\beta^\Gamma_p(V)-\beta^\Gamma_p(W)$ we use the following long exact sequence:

\begin{equation}\label{seq:Sigma-W-V}
\dots \to H_p(\Sigma;\K(\Gamma))\overset{i^+-i^-}\longrightarrow H_p(W;\K(\Gamma))\to H_p(V;\K(\Gamma))\to H_{p-1}(\Sigma;\K(\Gamma))\to \dots
\end{equation}
Let $K_p=\ker(H_p(\Sigma;\K(\Gamma))\overset{i^+-i^-}\longrightarrow H_p(W;\K(\Gamma)))$ to get the exact sequence:
\begin{equation*}
0\to K_p \to H_p(\Sigma;\K(\Gamma))\overset{i^+-i^-}\longrightarrow H_p(W;\K(\Gamma))\to H_p(V;\K(\Gamma))\to K_{p-1}\to 0
\end{equation*} 
so that if $k_p =\rank_{\K(\Gamma)}(K_p)$ then
\begin{equation}\label{eqn:Kp}\beta_p^\Gamma(V)-\beta^\Gamma_p(W) =k_p-\beta^\Gamma_p(\Sigma)+k_{p-1}.\end{equation}    Making this substitution and simplifying,  equation (\ref{eqn:rankDifference}) becomes:

\begin{align}\label{eqn:rankDifference2}
b(V)-b(W) =& \left(k_2-\beta^\Gamma_2(\Sigma)+k_1\right)+\left(k_1-\beta^\Gamma_1(\Sigma)+k_0\right)
 - \left(k_3 - \beta^\Gamma_3(\Sigma)+k_2\right)  - \left(k_0 -\beta^\Gamma_0(\Sigma)\right)  
 \\&+2 \left(\beta^\Gamma_2(\Sigma)-\beta^\Gamma_3(\Sigma)\right)
 \nonumber\\ =&
 2k_1-k_3+\beta^\Gamma_0(\Sigma)-\beta^\Gamma_1(\Sigma)+\beta^\Gamma_2(\Sigma)-\beta^\Gamma_3(\Sigma)= 2k_1-k_3.
 \nonumber
\end{align}
The final equality uses that the Euler characteristic of $\Sigma$ (a closed orientable three manifold) is zero.   Since $\beta^\Gamma_4(V)=\beta^\Gamma_4(W)=\beta^\Gamma_4(\Sigma)=0$ equation~(\ref{eqn:Kp}) reveals that $k_3=0$.  Finally, we see $b(V)-b(W)=2k_1\ge 0$ and $b(V)\ge b(W)$.  This completes the proof.  


\end{proof}

Next we suppose  that $V$ is an $n.5$-null-bordism and $\Gamma_r^{(n+1)}=0$.  In this setting the dual surfaces $D_i$ may not lift to the $\Gamma$-cover.  Nonetheless we get a similar result.

\begin{thm}\label{thm:n.5IsectionForm} Let $n\in \N$.  Suppose that $V$ is an $n.5$-null-bordism and $\psi:\pi_1(V)\to \Gamma$ is a homomorphism where $\Gamma$ is PTFA and $\Gamma_r^{(n+1)}=0$. Let $W = X(V)$ be the resulting cut open null-bordism.  Suppose additionally that either $\psi:\pi_1(\bdry  V)\to \Gamma$ is trivial, or that $\beta_1^\Gamma(\bdry V)= \beta_1(\bdry V)-1$.  Let $r$ be the $\K(\Gamma)$-rank of the span of  $\{\widetilde L_i\}$ as a subspace of $\frac{H_2(W;\K(\Gamma))}{H_2(\p W;\K(\Gamma))}.$  Then 
$
\rank_{\K(\Gamma)}\left(\frac{H_2(W;\K(\Gamma))}{H_2(\p W;\K(\Gamma))}\right) = 2r
$.
\end{thm}
\begin{proof}
By Proposition~\ref{highOrderSoln.5}, the set of lifts $\{\widetilde{L_1},\dots \widetilde{L_k}\}$ is linearly independent in $\frac{H_2(V;\K(\Gamma))}{H_2(\bdry V;\K(\Gamma))}$.  The non-degeneracy of the intersection form 
now implies that there exist classes $\delta_1,\dots \delta_k\in \frac{H_2(V;\K(\Gamma))}{H_2(\bdry V;\K(\Gamma))}$ dual to the $\widetilde L_i$'s.  That is $\delta_i\cdot \widetilde{L_j} = 1$ if $i=j$ and is 0 otherwise.      The set $\{\widetilde{L_1},\dots,\widetilde{L_k},\delta_1,\dots\delta_k\}$ is linearly independent, and so is a basis for $H_2(V;\K(\Gamma))/H_2(\bdry V;\K(\Gamma))$.  

 In order to get information about $W$ we study the exact sequence coming from (\ref{seq:Sigma-W-V}):
$$
 H_2(\Sigma; \K(\Gamma)) \overset{i_*^+ - i_*^-}{\longrightarrow} H_2(W; \K(\Gamma))\overset{j_*}\to B\to 0
$$
where 
$$B = \ker(\bdry_* :H_2(V; \K(\Gamma))\to H_1(\Sigma; \K(\Gamma))).$$
  Since the image of $(i_*^+-i_*^-)$ is carried by $\bdry W$ we see that any preimage under $j_*$ of a generating set of $B$ gives a generating set for $\frac{H_2(W;\K(\Gamma))}{H_2(\bdry W; \K(\Gamma))}$.  

Consider any $i$.  Since ${L_i}$ and ${\Sigma}$ are disjoint,  $\widetilde L_i$ is in the kernel of the map $\bdry_* :H_2(V; \K(\Gamma))\to H_1(\Sigma; \K(\Gamma))$, that is $\widetilde L_i\in B$.  Let $U_1 = \Span\{\widetilde{L_1},\dots \widetilde{L_k}\}\subseteq B$.  Let $U_2 = \Span\{\delta_1,\dots,\delta_k\} \cap B$ be the kernel of the restriction of  
$\bdry_*$ to the $\K(\Gamma)$-span of $\{\delta_1,\dots,\delta_k\}$.  Then 
$$B= \left(\im (H_2(\bdry V;\K(\Gamma))\to H_2(V;\K(\Gamma)))\right)\oplus U_1\oplus U_2$$ and any generating set for $U_1\oplus U_2$ produces a generating set for $\frac{H_2(W;\K(\Gamma))}{H_2(\bdry W;\K(\Gamma))}$.   

If $r$ is the $\K(\Gamma)$-Rank of $U_2$, then pick a basis $\{\epsilon_1,\dots,\epsilon_r\}$ for $U_2$.  Extend this to a full basis $\{\epsilon_1,\dots ,\epsilon_k\}$ for $\Span\{\delta_1,\dots,\delta_k\}\subseteq H_2(V;\K(\Gamma))$.  Pick a new basis $\{M_1,\dots M_k\}$ for $U_1$ which is dual to $\{\epsilon_1,\dots ,\epsilon_k\}$ with respect to the intersection form.  

Let $e_1,\dots e_r$ and $m_1,\dots m_k$ be preimages for $\epsilon_1,\dots \epsilon_r$ and $M_1,\dots M_k$ under the inclusion induced map $j_*:H_2(W;\K(\Gamma))\to H_2(V;\K(\Gamma))$.   Since $j_*$ respects the intersection form,  $e_i\cdot m_j = \epsilon_i \cdot M_j = 1$ if $i=j$ and is $0$ otherwise.  Thus, $e_1,\dots e_r, m_1,\dots m_r$ are linearly independent.  On the other hand if $j>r$ then $m_j$ pairs with every element of a generating set for $H_2(W;\K(\Gamma))$ to be zero under the intersection form.  By the non-singularity of the intersection form, then $m_j = 0$ in $\frac{H_2(W;\K(\Gamma))}{H_2(\bdry W;\K(\Gamma))}$ and $\{e_1,\dots,e_r, m_1,\dots, m_r\}$ forms a basis for $\frac{H_2(W;\K(\Gamma))}{H_2(\bdry W;\K(\Gamma))}$ and $\rank_{\K(\Gamma)}\left(\frac{H_2(W;\K(\Gamma))}{H_2(\p W;\K(\Gamma))}\right) = 2r$.  

\end{proof}

An important payoff of working with cut open null-bordisms is that the image of $\pi_1(W)\to \pi_1(V)$ is contained in $\pi_1(V)^{(1)}_r$.  If $\pi_1(V)\to \Gamma$ is a homomorphism where $\Gamma^{(n)}_r=0$ and we let $A\le \Gamma$ be the image of $\pi_1(W)$ then $A\le \Gamma^{(1)}_r$.  Thus, $A^{(n-1)}_r\subseteq \Gamma^{(n)}_r=0$.  Instead of studying a coefficient system on $V$ with $\Gamma^{(n)}_r=0$ we can study a coefficient system on $W$ with $A^{(n-1)}_r=0$.  


\begin{cor}\label{cor:IsectnFormCover}Let $n\in \N$ and $V$ be an $n$-null-bordism.  Let $W=X(V)$ be a cut open null-bordism.  Let $\psi:\pi_1(V)\to \Gamma$ be a homomorphism where $\Gamma^{(n)}_r=0$.  Suppose additionally that either $\psi:\pi_1(\bdry  V)\to \Gamma$ is trivial, or that $\rank_{\K(\Gamma)}( H_1(\bdry V;\K(\Gamma)))= \beta_1(\bdry V)-1$.  Suppose that $A\le \Gamma$ is a subgroup containing the image of $\pi_1(W)\to \Gamma$.  Then the collection of lifts $\{\widetilde L_i, \widetilde D_i\}$ forms a basis for
$
\frac{H_2(W;\K(A))}{H_2(\p W;\K(A))}.
$

If $V$ is instead an $n.5$-null-bordism and $\Gamma^{(n+1)}=0$ then the collection of lifts $\{\widetilde L_i\}\subseteq \frac{H_2(W;\K(A))}{H_2(\p W;\K(A))}$ spans a rank $r$ subspace where
$
\rank_{\K(A)}\left(\frac{H_2(W;\K(A))}{H_2(\p W;\K(A))}\right)=2r
$

\end{cor}

\begin{proof}
To see this one needs only notice that $\K(\Gamma)$ is flat over $\K(A)$, so that $H_p(W;\K(\Gamma))\cong H_p(W;\K(A))\otimes_{\K(A)} \K(\Gamma)$ and $H_p(\bdry W;\K(\Gamma))\cong H_p(\bdry W;\K(A))\otimes_{\K(A)} \K(\Gamma)$.  
\end{proof}

The reason we chose the letter $A$ for this group comes from the case $n=2$.  Under the map 
$$
\pi_1(V)\to \dfrac{\pi_1(V)}{\pi_1(V)^{(2)}_r}\into \Gamma:=\Z\ltimes \A(V)
$$
the image of $\pi_1(W)$ is contained in a finitely generated subgroup $A$ of the $\Q$-vector space $\A(V)$.  Under the hypothesis that $V$ is a non-degenerate null-bordism for $J$ then, the map $H_1(M_J)\to A$ is not the trivial map.  By studying invariants of $W$ corresponding to maps to the free abelian group $A$ we recover abelian invariants of $M_J$.  In the following section we make this explicit, using signatures of $W$ to compute the Levine-Tristram signature of $M(J)$.

\section{Cut open $1.5$-null-bordisms and signature invariants}\label{sect:sign}

First we recall the normalized Levine-Tristram signature function for a knot $J$, 
$\sigma_J:\T\to \Z$ where $\T = \{\omega\in \C:|\omega|=1\}$ is the unit circle in $\C$.  We will often just call $\sigma_J$ the signature function of $J$.

      Let $S$ be a Seifert matrix for $J$, for any $\omega\in\T$, the matrix $(1-\omega)S + (1-\omega^{-1})S^T$ is Hermitian.  The signature function of $J$ at $\omega$ is defined to be the signature of this matrix.  That is, the number of positive eigenvalues minus the number of negative eigenvalues (counted with multiplicity).  Away from the roots of $\Delta_J$ this gives a concordance invariant.  At the roots of $\Delta_J$, redefine $\sigma_J$ to be the average of its one-sided limits $$\sigma_J(e^{i\theta}) = \frac{1}{2}\left(\Lim_{t\to \theta^+} \sigma_J(e^{it}) +\Lim_{t\to \theta^-} \sigma_J(e^{it})\right).$$
      
  Suppose that $J$ is a knot and $V$ is a null-bordism for $M_J$.  Let $\mu_J$ be the meridian of $J$.  Since $\A(V)$ is torsion, $p(t)\cdot \mu_J=0$ for some $p(t)\in \Q[t,t^{-1}]$.  The \text{annihilator} of $\mu_J$ is the minimal degree polynomial with $p(t)\cdot \mu_J=0$.  

 In the case that the annihilator of $\mu_J$ has a factor in common with $at^{p}-b$ for some $a, b ,p\in \Z$ then we will be able to conclude that 
$$
\sigma_J(\omega^k) = g(\omega^{b})-g(\omega^a)
$$
 for some $k\in \N$ and some function $g:\T\to \Z$. In subsection \ref{subsect:Cooper}, we use this result to re-derive Cooper's signature condition (Theorem~\ref{Thm: Cooper sign 1}).   On the other hand if the annihilator has any factors which are coprime to $at^{p}-b$ for all $a, b ,p\in \Z$, then we will conclude remarkably that the signature function is identically zero.

Since our computation is done in terms of $\rho$-invariants we now recall some of their theory.  In \cite{APS2}, Atiyah, Patodi and Singer associate to a 3-manifold $M$ and a unitary representation $\alpha$ of $\pi_1(M)$ an integer, $\rho(M,\alpha)$.  We make use of a reformulation due to Casson and Gordon \cite{CG1, CG2}.  
 
 \begin{defn}\label{definition of unitary rho}
  For closed connected oriented 3-manifolds $M_1,\dots, M_k$, suppose that $W$ is a compact connected oriented  4-manifold with  $\bdry W = \underset{j=1}{\overset{k}{\sqcup}}M_j$.  Let  $\alpha:\pi_1(W)\to U(n)$ be a unitary representation of $\pi_1(W)$.  Let $\alpha_j = \alpha\circ (i_j)_*$ be the representation of $\pi_1(M_j)$ given by composition with the inclusion induced map $(i_j)_*:\pi_1(M_j)\to \pi_1(W)$.  Then
$$
\displaystyle\sum_{j=1}^k\rho(M_j,\alpha_j)=\sigma(W,\alpha)-n\cdot\sigma(W).
$$
\end{defn}
The untwisted signature of $W$, $\sigma(W)$ is given by taking the signature of a matrix for the intersection form on $H_2(W)/H_2(\bdry W)$.  The twisted signature $\sigma(W,\alpha)$ is given by doing the same with the twisted intersection form on $H_2^\alpha(W;\C^n)/H_2^\alpha(\bdry W;\C^n)$.

Importantly, according to \cite{Lith1}, the Levine-Tristram signature function can be understood as a $\rho$-invariant.  To be precise, for a knot $J$ and the representation $\alpha:\pi_1(M_J)\to U(1)$ which sends $\mu_J$ to the $1\times 1$ matrix $[\omega]$, $\sigma_J(\omega) = \rho(M_J, \alpha)$.

With the background now out of the way we state main result of this section:

\begin{thm}\label{thm: T*-1 sign}
Let $J$ be a knot.  Suppose that $V$ is a non-degenerate $1.5$-null-bordism for $M_J$.  Let $\mu_J$ be the class of the meridian of $J$ in the Alexander module, $\A(V)$.  Let $\delta(t)$ be an irreducible factor of the annihilator of $\eta$ in $\A(V)$.

\begin{enumerate}
\item\label{non-eigencase} In the case that $\delta(t)$ is relatively prime to  $at^p-b$ for all $a,b,p\in \Z-\{0\}$ then  for all $\omega\in\T$ $\sigma_J(\omega)=0$.
\item\label{eigencase} In the case that $\delta(t)$ divides $at^p-b$, there is a function $g:\T\to{\R} $ and nonzero integer $k$ such that for 
all $\omega\in\T$ 
$$\sigma_J(\omega^{k}) = 
g(\omega^{b}) - 
g(\omega^{a})$$
Moreover we may choose $g$  to be continuous away from finitely many algebraic numbers where $g$ has jump discontinuities.
\end{enumerate}
\end{thm}

Before we begin the proof, we apply it to prove Theorems~\ref{thm:genusonesign} and \ref{thm:doublingopsign}.  

\begin{proof}[Proof of Theorem~\ref{thm:genusonesign}]
Let $K$ be a genus one $1.5$-solvable knot with $\Delta_R(t) = \delta(t)\delta(t^{-1})$ for some $\delta(t) =  (mt-(m+1))$ with $m\notin\{0,-1\}$.  Suppose that the knot $J$ is a derivative associated to a $1.5$-solution for $K$. By Proposition~\ref{prop:NullBordismDeriv} $J$ admits a non-degenerate $1.5$-null-bordism $V$ such that the annihilator of $\mu_J$ has one of $\delta(t)$ and $\delta(t^{-1})\dot=((m+1)t-m)$ as a factor.  Case (2) of Theorem~\ref{thm: T*-1 sign} now applies and we see that there exists $k\in \N$ and a function $g:\T\to\R$ such that either 
$$
\sigma_J(\omega^k) = g(\omega^m)-g(\omega^{m+1})
\text{ or }
\sigma_J(\omega^k) = g(\omega^{m+1})-g(\omega^{m})
$$
replacing $g(\omega)$ by $-g(\omega)$ if needed completes the proof of Theorem~\ref{thm:genusonesign}.
\end{proof}

\begin{proof}[Proof of Theorem~\ref{thm:doublingopsign}]
Suppose that $R_\eta$ is a doubling operator such that for every isotropic submodule $Q$ there is a slice disk $\Delta$ for $R$ satisfying that $\eta\notin Q+P_\Delta$.  Suppose furthermore that $\Delta_R(t)$ is coprime to every polynomial of the form $at^p-b$.  

  Now suppose that $R_\eta(J)$ is $1.5$-solvable.  According to Proposition~\ref{Prop:NullBordismInfection} $M_J$ bounds a  non-degenerate 1.5-null-bordism $V$.  Moreover, as a class in $\A(V)$ the meridian of $J$ is annihilated by $\Delta_R$ which by assumption is coprime to $at^p-b$.  Thus, the annihilator has an irreducible factor coprime to $at^p-b$ for all $a,b,p\in \Z-\{0\}$.   Case (1) of Theorem~\ref{thm: T*-1 sign} now reveals that $\sigma_J(\omega)$ is identically zero.  
\end{proof}

\begin{proof}[Proof of Theorem~\ref{thm: T*-1 sign}]
Let $J$ be a knot.  Suppose  that $V$ is a non-degenerate $1.5$-null-bordism for $M_J$.  Since $V$ is non-degenerate, $\mu_J$ is non-zero in $\A(V)$.  Let $\delta(t)\in \Z[t,t^{-1}]$ be a irreducible factor of the annihilator of $\mu_J$ in $\A(V)$.  Since $\A(V)$ is a finitely generated module over the PID $\Q[t,t^{-1}]$, it has a cyclic summand $\A_\delta\cong\dfrac{\Q[t,t^{-1}]}{(\delta(t))^q}$ $q\in \N$ such that the projection map $\A(V)\to \A_\delta$ does not send $\eta$ to the zero element.  Let $\ell$ be the minimal positive integer such that $\delta(t)^\ell \eta=0$ in $\A_\delta$.  There is a generator $g$ for $\A_\delta$ such that $\delta(t)^{q-\ell} g =\eta$.  
 If $p$ is the degree of the polynomial $\delta$ then as a rational vector space, $\A_\delta$ has basis,
\begin{equation}\label{basis}\mathcal{B} = \left\{
t^i\delta(t)^j g: 0\le i<p,0\le j<q
\right\}.\end{equation}

Let $\Sigma$ be a 3-manifold dual to the generator of $H_1(V)$ and $W=X(V,\Sigma)$ be the resulting cut open $1.5$-null-bordism.  Notice that $W$ lifts to the free abelian cover of $V$.  Since $\pi_1(W)$ is finitely generated, it follows that the image of $\pi_1(W)\to \A(V)\to \A_\delta$ is finitely generated.  By replacing $g$ with $\dfrac{1}{k}\cdot g$ for some $k\in \N$ if necessary, we can assume that the image of $\pi_1(W)\to \A(V)\to \A_\delta$ is contained in the $\Z$-module generated by $\mathcal B$.  Let $A$ be that $\Z$-submodule.  
%
%
%
%
Observe then that if $\psi$ is the homomorphism $\pi_1(V)\to \dfrac{\pi_1(V)}{\pi_1(V)^{(2)}}\into \Z\ltimes \A(V) \to  \Z\ltimes \A_\delta$ then $\psi[\pi_1(W)]\subseteq A$.  According to Corollary~\ref{cor:IsectnFormCover}, for some $r\in \N$ $\rank_{\Q(A)}\left(\frac{H_2(W;\Q(A))}{H_2(\bdry W;\Q(A))}\right)=2r$ and there is a rank $r$ subspace on which the intersection form is zero.  We use the notation $\Q(A)$ in place of $\K(A)$ since $\Q[A]$ is commutative and the skew field completion agrees with the classical field of fractions.  


By reordering the basis $\mathcal{B}$ for $A$ above we get a basis $\{b_1,\dots b_n\}$ for which $k\cdot b_1 = \eta$.  For any ${\omega} = (\omega_1,\dots,\omega_n)\in \T^n$ (The $n$-fold cartesian product of $\T$ with itself) let $\phi_\omega:A\to U(1)$ be the map sending $b_i\mapsto [\omega_i]$.  (The $1\times 1$ matrix.)  If $\{\omega_1,\dots \omega_n\}$ is algebraically independent over $\Z$ then $\phi_\omega$ is injective and it is easy to see that $\phi_\omega$ extends to a map $\Q(A)\to \Gl(1)$. (The group of $1\times 1$ invertible matrices.) Since $\Gl(1)$ acts on $\C$, we may now regard $\C$ as a module over $\Q(A)$.  Modules over the field $\Q(A)$ are free and so are flat.  Thus, we see that:

\begin{prop}\label{prop: zero signature}
Let $W$ be a cut open $1.5$-null-bordism for $M_J$,  Provided that $\{\omega_1,\dots \omega_n\}\subseteq \T^n$ is  algebraically independent, the complex vector space $H_2^{\phi_\omega}(W;\C) \cong H_2(W;\Q(A))\otimes \C$ has complex dimension $2r$ for some $r\in \N$ and there is a $r$-dimensional subspace on which the intersection form vanishes.
\end{prop}

Thus, for some basis, the twisted  intersection form on $H_2^{\phi_\omega}(W;\C)$ is given by a $2r\times 2r$ non-singular Hermitian matrix with an $r\times r$ block of zeros.  As a consequence, the twisted signature is zero (See for example \cite[Chapter 8, Lemma 16]{R}.)  By Proposition~\ref{prop:IsectnForm} $\sigma(W)=0$.  Hence, the sum of the unitary $\rho$-invariants of components of the boundary of $W$ also vanishes.
$$
\rho\left(M_J, \phi_\omega\circ i_J\right) + 
\rho\left(\Sigma^+, \phi_\omega\circ i_+\right) + 
\rho\left(\Sigma^-, \phi_\omega\circ i_-\right) = 0
$$
where $i_J:\pi_1(M(J)) \to \pi_1(W) \to A$, $i_+:\pi_1(\Sigma^+) \to \pi_1(W) \to A$ and $i_-:\pi_1(\Sigma^-) \to \pi_1(W) \to A$ are the composition of the inclusion induced map with the map $\pi_1(W)\to A$.  The map $\phi_\omega$ sends $b_1$ to $\omega_1$ so that $\phi_\omega(\mu_J) = \phi_\omega(k\cdot b_1) = \omega_1^k$.  According to \cite{Lith1} this $\rho$-invariant of $M_J$ gives the Levine-Tristram signature:
\begin{equation}\label{eqn: vanish rho}
\sigma_J(\omega_1^k) = \rho\left(M_J,\phi_\omega\circ i_J\right) = -\rho\left(\Sigma^+, \phi_\omega\circ i_+\right) - 
\rho\left(\Sigma^-, \phi_\omega\circ i_-\right).
\end{equation}
If we want to understand the Levine-Tristram signatures of $J$ then it remains to understand the $\rho$-invariants of $\Sigma^+$ and $\Sigma^-$.  First recall that by their definition (as the positive and negative pushoffs of the 3-manifold $\Sigma$ dual to the generator of $H_1(V)$) there is an orientation reversing diffeomorphism $t:\Sigma^- \to \Sigma^+$ making the following diagram commute:
\begin{equation}
\begin{diagram}
\node{\pi_1(\Sigma^-)}\arrow{e,t}{i_-}\arrow{s,l}{t}\node{A}\arrow{e}\node{\A_\delta}\arrow{s,l}{t_*}\\
\node{\pi_1(\Sigma^+)}\arrow{e,t}{i_+}\node{A}\arrow{e}\node{\A_\delta}
\end{diagram}
\end{equation}
The map $t:\A_\delta\to \A_\delta$ is given by the $\Q[t^{\pm1}]$-module structure on $\A_\delta$.  It is possible that the free abelian group $A$ will not be closed under multiplication by $t$.  Indeed, consider the case that $\A_\delta = \dfrac{\Q[t,t^{-1}]}{(m+1)t-m}\cong \Q$.  ($m\notin \{0,-1\}$.)   On $\A_\delta$,  $t$ acts by multiplication by $\dfrac{m}{m+1}$.  Thus, if $A\cong \Z$ were closed under multiplication by $t$ then $\Z$ would be closed under multiplication by $\dfrac{m}{m+1}$, which is not true.


 In order to correct this let $d$ be any multiple of the leading coefficient of $\delta(t)$ and $T = d\cdot t$.  It is now easy to check that $T$ sends each element of the basis $\mathcal B$ of (\ref{basis}) to an integer linear combination of elements of $\mathcal B$.  Thus, $T$ sends $A$ to $A$ and we have the following commutative diagram, where $\times d$ indicates the map given by multiplication by $d\in \N$.
\begin{equation}
\begin{diagram}
\node{\pi_1(\Sigma^-)}\arrow{e,t}{i_-}\arrow{s,l}{t}\node{A}\arrow{se,t}{T}\\
\node{\pi_1(\Sigma^+)}\arrow{e,t}{i_+}\node{A}\arrow{e,t}{\times d} \node{A}\arrow{e,t}{\phi_\omega}\node{\C^X}
\end{diagram}
\end{equation}
Since $\rho$  is a diffeomorphism invariant it follows that
\begin{equation}\label{eqn: rho sigma}
-\rho\left(\Sigma^-, (\phi_\omega)\circ T\circ i_-\right) = \rho\left(\Sigma^+, (\phi_\omega)\circ (\times d)\circ i_+\right).
\end{equation}
Equations~(\ref{eqn: vanish rho}) and (\ref{eqn: rho sigma}) together imply that
$$
\sigma_J(\omega_1^{d\cdot k}) = \rho\left(\Sigma^-, \phi_\omega\circ T\circ i_-\right) - \rho\left(\Sigma^-, \phi_\omega\circ (\times d)\circ i_-\right)
$$

We now define a function $f:\T^n\to \R$ by $f(z) =  \rho\left(\Sigma^-, \phi_z\circ i_-\right)$.  According to \cite[Chapter II  Theorem 2.1]{L6} there is a stratification $\T^n = S_1\supseteq S_2\supseteq\dots \supseteq S_\ell = \emptyset$ where $S_1,S_2\dots S_\ell$ are algebraic sets and $f$ is continuous on $S_k-S_{k+1}$.  Moreover $f$ has only integer jumps, meaning $f$ is continuous when reduced mod $\Z$.

We will also need to let matrices act on $\T^n$.  For $\omega=(\omega_1,\dots\omega_n)$, $d\in \Z$ and  and $M = (m_{i,j})$ a matrix with integral entries we define $\omega^d = (\omega_1^d,\omega_2^d,\dots, \omega_n^d)$ and 
\begin{equation}
\label{exponential matrix}\omega^M = \left(\prod_{k=1}^n\omega_k^{m_{k,1}}, \prod_{k=1}^n\omega_k^{m_{k,2}},\dots , \prod_{k=1}^n\omega_k^{m_{k,n}}\right).
\end{equation}
The second definition comes from the identification of $\T^n$ with the quotient $\R^n/\Z^n$ on which integral matrices act by multiplication on the right.

\begin{lem} 
Let $M$ be a matrix for $T$ with respect to the basis $\mathcal{B}=\{b_1,\dots, b_n\}$.  Then for all $\omega\in \T^n$, 
$$
\rho(\Sigma^-, \phi_{ \omega}\circ (\times d)\circ i_-) = f(\omega^d)
$$
and 
$$
\rho(\Sigma^-, \phi_{\omega}\circ T\circ i_-) = f(\omega^{M}).
$$
\end{lem}
\begin{proof}

The proof follows by checking that as maps from $A$ to $\C^\times$,  $\phi_{\omega}\circ (\times d) = \phi_{\omega^d}$ and $\phi_{ \omega}\circ (T) = \phi_{\omega^M}$.  We verify the second equality.  Pick any basis element $b_j$ and let $\omega^M = z = (z_1,\dots, z_n)$, then 
$$
\phi_{\omega^M}(b_j) = \phi_z(b_j) = z_j = \prod_k\omega_k^{m_{k,j}}.
$$
On the other hand 
$$
\phi_{\omega}(T (b_j)) = \phi_\omega\left(\sum_k m_{k,j}b_k\right) = \prod_k \omega_k^{m_{k,j}}.
$$
These two maps are the same, completing the proof of the lemma.  The first equality is a special case of the second where $M=d\cdot \Id$ is a multiple of the identity matrix.
\end{proof}

Thus, assuming that $\{\omega_1,\dots, \omega_n\}$ is algebraically independent we see that 
\begin{equation}\label{eqn: signature almost done}
\sigma_J(\omega_1^{k\cdot d}) = f(\omega^M) -  f(\omega^d).
\end{equation}

In order to relate equation~(\ref{eqn: signature almost done}) with claim \ref{eigencase} of the theorem, we suppose  that $\A_\delta = \dfrac{\Q[t,t^{-1}]}{at-b}$ has $\Q$-rank $1$.  Then $\A_\delta$ has rational basis $\{\eta\}$ and for some $k$, $A$ is the integral span of $\frac{\eta}{k}$.  On $\A_\delta$ $t$ acts by multiplication by $\frac{b}{a}$ so that we can take $d=a$ and and $T = a\cdot t$ acts on $A$ by multiplication by $b$.  Finally, letting $\omega = \omega_1$ reduces equation~(\ref{eqn: signature almost done}) to
$$
\sigma_J(\omega^{k\cdot a}) = f(\omega^b) -  f(\omega^{a})
$$
and we have claim \ref{eigencase} of the theorem when $\omega$ is transcendental.  A continuity argument at the end of the proof extends this equality to algebraic numbers.

It remains to deal with the case that $\rank_\Q(\A_\delta) = \rank_\Z(A)$ is greater than $1$.  The following functional analytic lemma allows us to translate equation~(\ref{eqn: signature almost done}) into a one variable result.  The proof of the lemma is technical, involving the theory of operators on the Hilbert space of square integrable functions on $\T^n$.  It is delayed until after we complete the proof of Theorem~\ref{thm: T*-1 sign}. 


 \begin{lem}\label{lem: analysis lemma}
Let $s:\T\to \R$ and $f:\T^n\to \R$ be bounded measurable functions, $d\in \Z-\{0\}$ and $M$ be an $n\times n$ integral matrix with nonzero determinant.  Suppose that for almost all $\omega\in \T^n$
 \begin{equation}\label{signEqn1}
s(\omega_1) = 
f(\omega^M) - 
f(\omega^d).
 \end{equation}
 Let $e_1=(1,0,0,\dots)^T$ be the first standard basis element of $\R^n$.
 
 \begin{enumerate}
 \item If for all $p\in \N$ it is the case that $e_1$ is not an eigenvector of $M^p$ then $s(\omega_1)=0$ almost everywhere.
 \item  Suppose that for all $q<p$  $e_1$ is not an eigenvector for $M^q$ and that $M^p e_1 = \lambda e_1$ for some $\lambda\in \Z$.  
 Then there is a function $g:\T\to \R$ such that for almost every $\omega_1\in \T$
 \begin{equation}
s(\omega_1^{d^{p-1}}) = 
g(\omega_1^\lambda) - 
g(\omega_1^{d^{p}}).
 \end{equation}
   If there is a proper algebraic subset $S$ of $\T^n$ such that $f$ is continuous on $\T^n-S$  then $g$ can be chosen to be continuous away from a finite collection of algebraic numbers.  If $f$ has only jump discontinuities, then the same is true of $g$.  
  \end{enumerate}
 \end{lem} 
 
 Now we complete the proof of Theorem~\ref{thm: T*-1 sign}.  Let $M$ be a matrix representative of $T=d\cdot t:H\to H$.   According to equation~(\ref{eqn: signature almost done}), 
$\sigma_J(\omega_1^{d\cdot k}) = f(\omega^M) -  f(\omega^d)$ for almost all $\omega\in \T^n$.  Recall that we chose a basis $b_1,\dots, b_n$ with $k\cdot b_1=\eta$

In the first case of the theorem $(at^p-b)\eta\neq0$ for all $a, b, p$.  It follows that for all $p\in \N$,  $b_1$ is not an eigenvector of $T^p$, and $e_1$ is not an eigenvector for $M^p$.   Case (1) of Lemma~\ref{lem: analysis lemma} applies and $\sigma_J(\omega_1)=0$ for almost all $\omega_1\in \T$.  Since $\sigma_J$ has only jump discontinuities where it is normalized to evaluate to the average of the one sided limits, it follows that $\sigma_J(\omega)=0$ for all $\omega\in \T^n$  This recovers conclusion (1) of the Theorem.

On the other hand, if $\delta(t)$ is a factor of $at^p-b$ then $a$ is a multiple of the leading coefficient of $\delta$.  Take $d=a$ so $T=a\cdot t$ and $$T^p b_1 = a^p t^p b_1 = a^p \dfrac{b}{a} b_1 = ba^{p-1}b_1.$$ Thus, $b_1$ is an eigenvector for $T^p$ with eigenvalue $\lambda = b\cdot a^{p-1}$.  By Lemma~\ref{lem: analysis lemma} there is a function $g:\T\to \R$ such that  $\sigma_J(\omega^{k\cdot a \cdot a^{p-1}}) = 
g(\omega^{b\cdot a^{p-1}}) - 
g(\omega^{a^{p}})$  almost everywhere. At the discontinuities of $g$ (which are jump discontinuities, so one-sided limits exist) we may redefine $g$ as the average of its one-sided limits.  Hence, for all $\omega\in \T$, \begin{equation}\label{before simplification}
\sigma_J(\omega^{k\cdot a^{p}}) = 
g(\omega^{b\cdot a^{p-1}}) - 
g(\omega^{a^{p}})
\end{equation}
For any $z\in \T$ let $\omega = z^{(a^{1-p})}$.  Making this substitution into  (\ref{before simplification}) we see that 
 \begin{equation}\label{eqn: desired equality}
\sigma_J(z^{k\cdot a}) = 
g(z^{b}) - 
g(z^{a})
\end{equation}
 for all $z\in \T$.    Replacing $k\cdot a\in \N$ with the symbol $k$ completes the proof of Theorem~\ref{thm: T*-1 sign}.

\end{proof}

 \begin{proof}[Proof of Lemma~\ref{lem: analysis lemma}]

Since the assumed equality $s(\omega_1) = 
f(\omega^M) - 
f(\omega^d)$ holds almost everywhere, we may regard it as an equality in the Hilbert space $L^2(\T^n)$ of square integrable functions on $\T^n$ with respect to normalized Lebesgue measure on $\T^n$, that is, the complete inner product space with inner product 
$$
\langle f, g\rangle = \frac{1}{(2\pi)^n}\Int_{\T^n} f(\omega_1,\dots,\omega_n)\overline{g(\omega_1,\dots, \omega_n)}d\omega_1\dots d\omega_n
.$$  
Using this inner product, the space $L^2(\T^n)$ has an orthogonal basis consisting of monomials.
$$
\{\omega_1^{i_1}\cdot \ldots\cdot \omega_n^{i_n} : i_1,\dots, i_n\in \Z\}.
$$
For any column vector $I=(i_1,\dots, i_n)^T\in \Z^n$, define $\omega^I := \omega_1^{i_1}\cdot \ldots \cdot \omega_n^{i_n}$.  The $T$ in the superscipt indicates the transpose.  

Notice that for $d\in \Z$, $(\omega^d)^I = \omega_1^{d\cdot i_1}\cdot\ldots\cdot \omega_n^{d\cdot i_n} = \omega^{dI}$.  Similarly, for a matrix $M$, 
$$
\begin{array}{rcll}
(\omega^M)^I
&=&
\displaystyle \left(\prod_{k=1}^n\omega_k^{m_{k,1}}, \prod_{k=1}^n\omega_k^{m_{k,2}},\dots , \prod_{k=1}^n\omega_k^{m_{k,n}}\right)^I
 &\text{by equation }(\ref{exponential matrix})
\\&=&
\displaystyle\left(\prod_{k=1}^n\omega_k^{m_{k,1}\cdot i_1}\right)\cdot  \left(\prod_{k=1}^n\omega_k^{m_{k,2}\cdot i_2}\right)\cdot \ldots \cdot \left(\prod_{k=1}^n\omega_k^{m_{k,n}\cdot i_n}\right)
\\&=&
\displaystyle\left(\omega_1^{ \underset{j=1}{\overset{n}{\sum}} m_{1,j}i_j}\right)\cdot\left(\omega_2^{\underset{j=1}{\overset{n}{\sum}}m_{2,j}i_j}\right)\cdot\ldots\cdot
\left(\omega_n^{\underset{j=1}{\overset{n}{\sum}} m_{n,j}i_j}\right)
\\&=&
\omega^{MI}.  
\end{array}
$$
  Suppose $f(\omega) = \Sum_{I\in \Z^n} a_I\cdot \omega^I$ and $s(\omega_1)=\Sum_{k\in \Z} s_k \omega_1^k$.  Since $s(\omega_1) = 
f(\omega^M) - 
f(\omega^d)$, we have that
\begin{equation}\label{difference 1}
\Sum_{k\in \Z} s_k \omega_1^k = \sum_I a_I\cdot \omega^{MI} - \sum_I a_I\cdot \omega^{dI}.
\end{equation}
Now, two sums of basis elements are the same if and only if the coefficients agree.  Thus, for all $k\in \Z$, if $I=(k,0,\dots, 0)^T$ then 
\begin{equation}\label{difference 2}
 s_k = a_{M^{-1}I} - a_{d^{-1}I}
\end{equation}
and if $I$ is not a multiple of $(1,0,\dots,0)^T$ then 
\begin{equation}\label{difference 2a}
 0 = a_{M^{-1}I} - a_{d^{-1}I}
\end{equation}
These equations hold even if some of the subscripts are in $\Q^n$ rather than $\Z^n$, as long as we make the convention that $a_J=0$ when $J\notin \Z^n$ and $s_k=0$ when $k\notin \Z$.   This is equivalent to noticing that $L^2(\T^n)\cong \ell^2(\Z^n)\into \ell^2(\Q^n)$ and instead regarding the original equality in $\ell^2(\Q^n)$.    

If $k=0$ then equation~(\ref{difference 2}) gives 
$
s_{0} = a_{0}-a_{0} = 0.
$
Let $J = (k,0,\dots 0)$ with $k\neq 0$.  Assume for all $p\in \N$ that $e_1$ is not an eigenvalue of  $M^p$ so that $M^{-p}J$ is not a multiple of  $(1,0,\dots,0)^T$.  Then 

\begin{eqnarray*}
a_{M^{-1}J} - a_{d^{-1}J} =& s_k&\text{by (\ref{difference 2}) using }I=J\\
a_{d^{}M^{-2}J} - a_{M^{-1}J} =& 0&\text{by (\ref{difference 2a}) using }I=dM^{-1}J\\
a_{d^{2}M^{-3}J} - a_{dM^{-2}J} =& 0&\text{by (\ref{difference 2a}) using }I=d^2M^{-2}J\\
\vdots\\
a_{d^{n}M^{-n-1}J} - a_{d^{n-1}M^{-n}J} =& 0&\text{by (\ref{difference 2a}) using }I=d^nM^{-n}J\\
\end{eqnarray*}
So that $a_{M^{-1}J} = a_{d^{}M^{-2}J} = \dots = a_{d^{n}M^{-n-1}I}$ for all $n\in \N$.   Let $v = a_{M^{-1}J} $ be this value.  Since $J$ is not an eigenvector for $M^p$ for any $p$, ${M^{-n}d^{n-1}J}\neq {M^{-{m}-1}d^{m}J}$ for all $n\neq m$ and we see that $|v|^2$ appears infinitely many times in the convergent sum $\Sum_{I\in \Q^n} |a_I|^2 = ||f(\omega)||^2$. Thus, $v = a_{M^{-1}J}$ must be zero.  A similar argument shows that $a_{d^{-1}J} = 0$ and equation~(\ref{difference 2}) concludes that $s_k=0$ for all $k\neq 0$.

Thus, $s=0$ in $L^2(\T)$ and hence vanishes almost everywhere.  This completes the proof of the first claim.

Now suppose that $M^p e_1 = \lambda e_1$ and $e_1$ is not an eigenvector for $M^q$ for any $q<p$.  Similarly to before, if $J=(k,0,\dots,0)^T$ with $k\neq 0$ then 
\begin{eqnarray*}
a_{M^{-1}J} - a_{d^{-1}J} =& s_k& \text{by (\ref{difference 2}) using }I=J\\\
a_{d^{}M^{-2}J} - a_{M^{-1}J} =& 0& \text{by (\ref{difference 2a}) using }I=dM^{-1}J\\
a_{d^{2}M^{-3}J} - a_{dM^{-2}J} =& 0& \text{by (\ref{difference 2a}) using }I=dM^{-2}J\\
\vdots\\
a_{d^{p-1}M^{-p}J} - a_{d^{p-2}M^{-p+1}J} =& 0& \text{by (\ref{difference 2a}) using }I=d^{p-1}M^{-p+1}J
\end{eqnarray*}
so that 
$$a_{M^{-1}J} = a_{d^{}M^{-2}J} = a_{d^{2}M^{-3}J} = \dots = a_{d^{p-1}M^{-p}J}.$$
By assumption $M^p e_1 = \lambda e_1$ so that $M^{-p} J = \lambda^{-1} J$.  Making this substitution  
$$a_{M^{-1}J} =  a_{d^{p-1}\lambda^{-1}J}$$
  Thus, if $J=(k,0,0,\dots)^T$, then equation (\ref{difference 2}) reduces to 
$$a_{d^{p-1}\lambda^{-1}J} - a_{d^{-1}J} = s_k.$$
Take $a_k$ to mean $a_{(k,0,\dots,0)^T}$.  Then
$$
s(\omega_1) = \Sum_{k\in \Z} s_k \omega_1^k = \Sum_{k\in \Q} s_k \omega_1^k = \Sum_{k\in \Q} \left(a_{d^{p-1}\lambda^{-1}k} - a_{d^{-1}k}\right)\omega_1^k.
$$
Plugging in $\omega_1^{(d^{p-1})}$ for $\omega_1$  gives an isometric embedding $L^2(\T)\into L^2(\T)$ and an isometric isomorphism $\ell^2(\Q)\overset{\cong}\to \ell^2(\Q)$.  Applying this map, we see
$$
s\left(\omega_1^{(d^{p-1})}\right) =\Sum_{k\in \Q} \left(a_{d^{p-1}\lambda^{-1}k} - a_{d^{-1}k}\right)\omega_1^{(d^{p-1})k}.
$$
Breaking up the sum on the right hand side and re-indexing gives
$$
s\left(\omega_1^{(d^{p-1})}\right) =\Sum_{k\in \Q} a_k \omega_1^{\lambda k}-\Sum_{k\in \Q} a_k \omega_1^{d^p k}.
$$
Since $f(\omega) = \Sum_{I\in \Z^n} a_I\omega^I$, $a_k=0$ if $k\notin\Z$.  Thus,
$$
s\left(\omega_1^{(d^{p-1})}\right) =\Sum_{k\in \Z} a_k \omega_1^{\lambda k}-\Sum_{k\in \Z} a_k \omega_1^{d^p k}.
$$
Finally, $g(\omega_1) = \Sum_{k\in \Z} a_k \omega_1^k$ gives the desired result.  It remains to to study continuity properties of $g$.  The following lemma expresses $g$ as an integral of $f$.  

\begin{lem}
For any $ f(\omega) = \Sum_{I\in \Z^n} a_I \omega^I\in L^2(\T^n)$
\begin{equation}\label{operator lemma}\Sum_{k\in \Z} a_{(k,0,\dots,0)} \omega_1^k = \frac{1}{(2\pi)^{n-1}}\Int_{\T^{n-1}} f(\omega_1,\omega_2,\dots\omega_n) ~ d\omega_2,\dots d\omega_n \end{equation}
\end{lem}
\begin{proof}
We verify the equality when $f(\omega) = \omega^I = \omega^{i_1}\cdot \ldots \omega_n^{i_n}$.  The right hand side of (\ref{operator lemma}) expands as 
$$
 \frac{1}{(2\pi)^{n-1}}\Int_{\T^{n-1}} \omega_1^{i_1}\cdot \ldots \omega_n^{i_n} ~ d\omega_2,\dots d\omega_n 
 = \omega_1^{i_1}\cdot  \left(\frac{1}{2\pi}\Int_{\T^{1}} \omega_2^{i_n}d\omega_2\right)\cdot \ldots \cdot \left(\frac{1}{2\pi}\Int_{\T^{1}} \omega_n^{i_n}d\omega_n\right)
$$ 
Recall that $\dfrac{1}{2\pi}\Int_{\T^{1}} z^{k}dz = \left\{\begin{array}{ll} 1&\text{ if }k=0\\0& \text{ if }k\neq0\end{array}\right.$.  Thus
$$
 \frac{1}{(2\pi)^{n-1}}\Int_{\T^{n-1}} \omega_1^{i_1}\cdot \ldots \omega_n^{i_n} ~ d\omega_2,\dots d\omega_n = \left\{\begin{array}{ll} \omega_1^{i_1}& \text{ if }i_2=i_3=\dots=i_n=0\\0&\text{ otherwise }\end{array}\right.
$$ 
This agrees precisely with the left hand side of (\ref{operator lemma}) for $f(\omega) = \omega^I$.  Since the rules sending $f(\omega)\in L^2(\T^n)$ to  the right and left hand side of (\ref{operator lemma}) are bounded linear operators and these operators agree on a basis for $L^2(\T^n)$,  the claim holds for all $f(\omega)\in L^2(\T^n)$.  
\end{proof}

  We now see that
  $$
  g(\omega_1) = \frac{1}{(2\pi)^{n-1}}\Int_{\T^{n-1}} f(\omega_1,\omega_2,\dots\omega_n) ~ d\omega_2,\dots d\omega_n.
  $$
  We must now check that $g$ is continuous away from a finite set of algebraic numbers, where it has jump discontinuities.  

Let $S\subseteq \T^n$ be the proper algebraic set away from which $f$ is continuous.  $S$ decomposes into irreducible components $S=X_1\cup X_2\cup\dots \cup X_n$.  By reordering the $X_k$ we assume that for $k=1,\dots,j$ there is some $z_k\in \T$ such that $X_k = \{z_k\}\times \T^{n-1}$ and that for $z\notin\{z_1,\dots z_k\}$ $S\cap (\{z\}\times \T^{n-1})$ is a proper algebraic subset of $\{z\}\times \T^{n-1}$ and so has zero $n-1$ dimensional measure.

Let $z\notin \{z_1,\dots z_k\}$  then away from a measure zero set of $\T^{n-1}$ we have that 
$$
\Lim_{\omega_1\to z} f(\omega_1,\omega_2,\dots \omega_n) = f(z,\omega_2,\dots \omega_n)
$$
Since $f$ is bounded above by some constant and constant functions on $\T^{n-1}$ are integrable, it follows from the domainated convergence theorem (See for example, \cite[Theorem 1.34]{Rudin})  that 
$$
\Lim_{\omega_1\to z} g(\omega_1) = g(z) 
$$
so that $g$ is continuous at $z$.

Suppose now that $f$ has only jump discontinuities.  It remains to show that $g$ has (at worst) a jump discontinuity at $z_j = e^{i\theta_j}$ when $X_j = \{z_j\}\times \T^{n-1}$ is an irreducible component of $S$.  We must show that the one sided limits exist at $z_j$. Since $f$ has only jump discontinuities we see that for $\epsilon\in \{+,-\}$ the one sided limit
$\Lim_{\theta\to \theta_j^{\epsilon}}f(e^{i\theta},\omega_2,\dots \omega_n)$ exists.  Using dominated convergence again we see that
$$
\begin{array}{rcl}\Lim_{\theta\to \theta_j^{\epsilon}}g(e^{i\theta}) &= &
\Lim_{\theta\to\theta_j^\epsilon} \left(\dfrac{1}{(2\pi)^{n-1}}\Int_{\T^{n-1}} f(\omega_1,\omega_2,\dots\omega_n) ~ d\omega_2,\dots d\omega_n\right)
\\
&=& 
\dfrac{1}{(2\pi)^{n-1}} \Int_{\T^{n-1}} \left(\Lim_{\theta\to\theta_j^\epsilon} f(\omega_1,\omega_2,\dots\omega_n)\right) ~ d\omega_2,\dots d\omega_n
\end{array}
$$  In particular, $\Lim_{\theta\to \theta_j^{\epsilon}}g(e^{i\theta})$ exists and $g$ has only jump discontinuities.

 \end{proof}

\subsection{Comparison with the Cooper signature condition}\label{subsect:Cooper}

As an application we recover the signature condition appearing in Cooper's thesis \cite[Corollary 3.14]{CooperThesis}.  We phrase the result using language of Gilmer and Livingston \cite{GL2}.

\begin{defn}[Definition 2 of \cite{GL2}]
Let $m,n\in \N$.  A function $f:\T\to \Z$ satisfies the Cooper $(m,n)$-signature condition if for all $p,c\in \N$ such that $(m,p)=(n,p)=(c,p)=1$, it follows that
$$
\displaystyle \sum_{\ell=1}^{r} f\left( e^{2\pi i \frac{c (n \overline{m})^\ell}{p}} \right) = 0
$$
where $\overline{m}\in \N$ is an inverse for $m$ mod $p$ and $r$ is the order of $n\cdot \overline{m}$ in the multiplicative group of units in $\Z/p$.
\end{defn}

 Cooper uses Casson-Gordon invariants to show that if $K$ is a genus one slice knot with Alexander polynomial $\Delta_K = ((m+1)t-m)(mt-(m+1))$ then on any genus one Seifert surface there is a derivative whose signature function satisfies the $(m,m+1)$-signature condition.  According to Proposition~\ref{prop:NullBordismDeriv} if $K$ is genus one  and slice (or even just $1.5$-solvable) and has nonzero Alexander polynomial $\Delta_K = ((m+1)t-m)(mt-(m+1))\neq 1$, then for any genus one Seifert surface there is a derivative $J$ such that $M_J$ has a non-degenerate $1.5$-null-bordism, $V$.  Notice that the annihilator of $\mu_J$ in $\A(V)$ has at least one of  $((m+1)t-m)$ and $(mt-(m+1))$ as a factor.  The second result of Theorem~\ref{thm: T*-1 sign} applies and and there exists a nonzero integer $k$ and a function $g$ such that 
\begin{equation*}
\sigma_J(\omega^{k}) = 
g(\omega^{m}) - 
g(\omega^{(m+1)})
\end{equation*}
We now prove that this equality implies Cooper's signature condition.

\begin{prop}\label{prop:CooperCondition1}
Let $m$ and $n$ be coprime.  Let $f$ and $g$ be functions on $\T$.  If there is an integer $k$ such that $f(\omega^k) = g(\omega^{m}) -  g(\omega^n) $ then $f$ has vanishing Cooper $(m,n)$-signature condition
\end{prop}

\begin{proof}

Suppose that $f(\omega^k) = g(\omega^{m}) -  g(\omega^n) $ for some function $g$ and consider any $p,c\in \N$ such that $(p,m)=(p,n) =(p,c)= 1$.  For the sake of notational convenience let $F(x) = f(e^{2\pi i x})$ and $G(x)  = g(e^{2\pi i x})$.  Then $F(k\cdot x) = G(m\cdot x)-G(n\cdot x)$. Let $\overline{m}$ be an inverse to $m$ mod $p$, $r$ be the order of $n\cdot \overline{m}$ mod $p$ and 
\begin{equation}\label{left hand side}
S=\displaystyle \sum_{\ell=1}^{r}F\left( \frac{c (n\cdot \overline{m})^\ell}{p}\right).
\end{equation}
We need to show that $S=0$.

 The integer $k$ decomposes as $k_0\cdot k_1$ where $(k_0,p)=1$ and every prime factor of $k_1$ is a prime factor of $p$.  Thus, it is also true that $(p\cdot k_1,k_0) = (p\cdot k_1,m)=(p\cdot k_1,n) = 1$.    Replacing $\overline{m}$ with a different inverse to $m$ mod $p$ will not change $S$, since $F$ is well defined on $\R/\Z$.  Thus, we may assume that $\overline{m}$ is an inverse of $m$ mod $(p\cdot k_1)$.  

Let $\overline{k_0}$ be an inverse for $k_0$ mod $p\cdot k_1$. If $r$ is the multiplicative order of $n\cdot \overline{m}$ mod p, then the order of $n\cdot \overline{m}$ mod $(p\cdot k_1)$ will be a multiple of $r$, say $q\cdot r$  (with $q\in \N$).  The value of $(n \overline{m})^\ell$ in the sum for $S$ in (\ref{left hand side}) runs over all powers of $n\cdot \overline{m}$ mod $p$.  If we instead let the index run from $1$ to $q\cdot r$ then $(n\cdot \overline{m})^\ell$ will run over all such powers $q$ times, Thus:
$$
\begin{array}{rcl}
q\cdot S &=&  
\displaystyle \sum_{\ell=1}^{q\cdot r}F\left( \frac{c (n\cdot \overline{m})^\ell}{p}\right) 
=
\displaystyle \sum_{\ell=1}^{q\cdot r}F\left(k_0\cdot k_1\frac{\overline{k_0}\cdot c \cdot  n^{\ell}\cdot \overline{m}^\ell}{p\cdot k_1}\right) 
=
\displaystyle \sum_{\ell=1}^{q\cdot r}F\left(k\cdot\frac{\overline{k_0}\cdot c \cdot  n^{\ell}\cdot \overline{m}^\ell}{p\cdot k_1}\right).
\end{array}
$$
The second equality above comes from multiplying and dividing by $k_1$ and then multiplying by $k_0\overline{k_0}$ which is congruent to $1$ mod $p\cdot k_1$.
By assumption, for all $x\in \R$, $F(k\cdot x) =  G(m\cdot x) - G(n\cdot x)$.  Applying this fact to the formula for $q\cdot S$ above we get
$$
\begin{array}{rcl}
q\cdot S &=&
\displaystyle \sum_{\ell=1}^{q\cdot r}G\left(m\cdot\frac{\overline{k_0}\cdot c \cdot  n^{\ell}\cdot \overline{m}^\ell}{p\cdot k_1}\right) -
\displaystyle \sum_{\ell=1}^{q\cdot r}G\left(n\cdot\frac{\overline{k_0}\cdot c \cdot  n^{\ell}\cdot \overline{m}^\ell}{p\cdot k_1}\right)
\\&=&
\displaystyle \sum_{\ell=1}^{q\cdot r}G\left(\frac{\overline{k_0}\cdot c \cdot  n^{\ell}\cdot \overline{m}^{\ell-1}}{p\cdot k_1}\right) -
\displaystyle \sum_{\ell=1}^{q\cdot r}G\left(\frac{\overline{k_0}\cdot c \cdot  n^{\ell+1}\cdot \overline{m}^\ell}{p\cdot k_1}\right)
\end{array}
$$
All but the first term of the leftmost sum and last term of the rightmost sum cancel, so that
$$
\begin{array}{rcl}
q\cdot S &=&
\displaystyle
G\left(\frac{\overline{k_0}\cdot c \cdot  n^{1}\cdot \overline{m}^{0}}{p\cdot k_1}\right) -
G\left(\frac{\overline{k_0}\cdot c \cdot  n^{q\cdot r+ 1}\cdot \overline{m}^{q\cdot r}}{p\cdot k_1}\right)\end{array}
$$
Since $q\cdot r$ is the order of $n\cdot\overline{m}$ mod $p\cdot k_1$ these terms also cancel, so that $q\cdot S=0$, and so $S=0$, proving the result.
\end{proof}

\section{A sufficient condition for a genus one algebraically slice knot to be $1.5$-solvable.}\label{sect:suff}

In this section we prove Theorem~\ref{thm:sufficiency} providing a sufficient condition for a genus one algebraically slice knot to be $1.5$-solvable.  For convenience we reiterate the result here.

\begin{thmnn}[Theorem~\ref{thm:sufficiency}]
Let $K$ be a genus one algebraically slice knot with $\Delta_K(t)= (mt-(m+1))((m+1)t-m)$, where $m\notin \{-1,0\}$.  Suppose that $K$ has a derivative knot $J$.  If  the algebraic concordance type of $J$ satisfies
$
[J]=[T_{(m,1)}]-[T_{(m+1,1)}]
$
for some knot $T$, then $K$ is $1.5$-solvable via a $1.5$-solution to which $J$ is associated.
\end{thmnn}

Our proof requires a generalization of the classical satellite operation.  One new feature is that this operation  produces links in homology spheres which might not be $S^3$.   

\begin{defn}[Knotted satellite operators] \label{defn:genSatOp}
Let $P$ be a knot or link in the homology sphere $M$ and $\eta$ be a framed curve in the complement of $P$.  Given any  knot in a homology sphere $J\subseteq Y$, one can build a new 3-manifold by cutting out a tubular neighborhood of $\eta$ and gluing back in the exterior of $J$ so that the meridian of $\eta$ is identified with the zero-framed longitude of $J$ and meridian of $J$ is identified to the framed longitude of $\eta$.  This results in a (possibly new) 3-manifold $M_{\eta}(J)$.  Since the meridian of $\eta$ is sent to the zero-framed longitude of $J$,   $M_\eta(J)$ is a homology sphere.  $P_{\eta}(J)$ is the image of $P$ in  $M_\eta(J)$.  We call $P_\eta$ a \textbf{knotted satellite operator}.
\end{defn}

If $\eta_1$ and $\eta_2$ are disjoint framed curves in the complement of $P$ and $J_1$ and $J_2$ are knots in homology spheres then one can iteratively perform this construction, generating the knot (or link) $P_{\eta_1,\eta_2}(J_1,J_2)$.  We call $P_{\eta_1,\eta_2}$ a \textbf{linked satellite operator}.

\begin{figure}[htbp]
\setlength{\unitlength}{1pt}
\begin{picture}(200,110)
\put(0,0){\includegraphics[height=110pt]{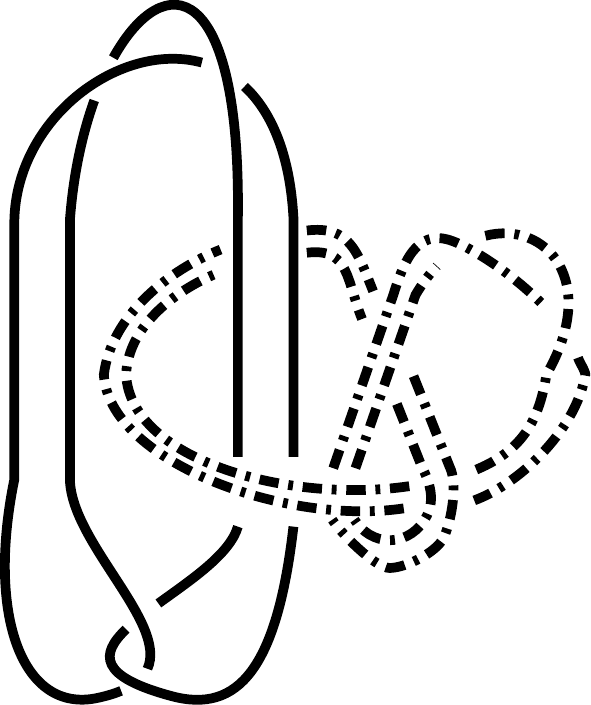}}
\put(47, 0){$P$}
\put(95, 50){$\eta$}

\put(120,0){\includegraphics[height=110pt]{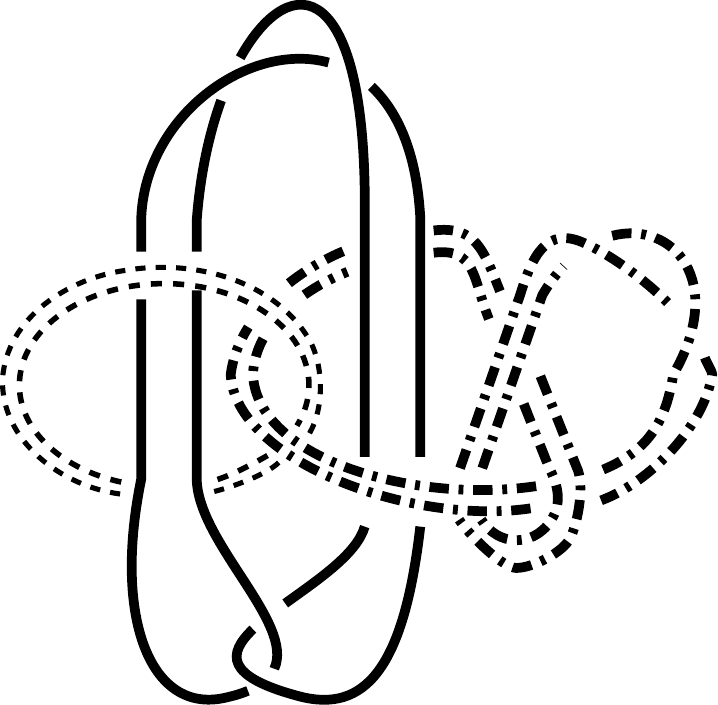}}
\put(185, 0){$P$}
\put(235, 50){$\eta_2$}
\put(125, 50){$\eta_2$}
\end{picture}
\caption{Left: An example of a knotted satellite operator $P_{\eta}$.  Right: A linked satellite operator $P_{\eta_1,\eta_2}$.  To form the resulting satellite knot $P_\eta(J)$ cut out a neighborhood of $\eta$.   Glue in the complement of $J$  by identifying the longitude of $J$ with the meridian of $\eta$ and the chosen pushoff of $\eta$ with the meridian of $J$.}\label{fig:genSatOp}
\end{figure}
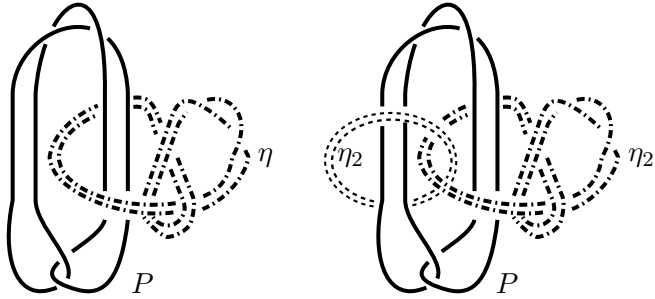

If $\eta$ were unknotted in the complement of $P$ then the classical satellite $P_\eta(J)$ is given by cutting out a neighborhood of $\eta$ and gluing in the exterior of $J$ so that the meridian of $J$ is identified with the {zero-framed} longitude of $\eta$.  See for example \cite[Section 3]{CD1}.  It is clear that Definition~\ref{defn:genSatOp} specializes to the classical satellite when $\eta$ is a zero-framed unknot.  As the following lemma reveals,  the knotted satellite operation preserves linking number.

\begin{lem}\label{lem:genSatOpLinking}
Let $P_\eta$ be a knotted satellite operator where $P=P_1\sqcup P_2$ is a two component link.  For any knot $J$, let $Q =P_\eta(J)$.  Then  $\lnk(P_1,P_2) = \lnk(Q_1,Q_2)$.
\end{lem}
\begin{proof}
Suppose $(P_1, P_2, \eta)$ is a link in the homology sphere $M$.  Let $F$ be a Seifert surface for $P_1$.  The linking number $\lnk(P_1,P_2)$ is given by counting with sign the number of intersections of $F$ with $P_2$.  Suppose that $F$ intersects $\eta$ in some points $c_1,\dots c_n$.  To build a Seifert surface for $Q_1$ begin by cutting out neighborhoods of these intersections points.  This results in a surface $F_0$ bounded by $F$ together with some meridians of $\eta$.  In the complement of $Q = P_\eta(J)$, each of these bound disjoint Seifert surfaces for $J$ in $M_\eta(J)$.  Glue these surfaces to $F_0$ along these meridians for $\eta$ to get the Seifert surface $F'$ for $Q_1=(P_1)_\eta(J)$. Since the Seifert surface for $J$ does not intersect $Q_2$ the intersection between $F'$ and $Q_2$ is the same as the intersection between $F$ and $P_2$.  This completes the proof.  
\end{proof}
  
  Two knots in homology 3-spheres $R\subseteq M$ and $R'\subseteq M'$ are called \textbf{homology concordant} if there is a homology cobordism from $M$ to $M'$ in which $R$ and $R'$ cobound a smoothly embedded annulus.     The following proposition generalizes \cite[Section 3]{CD1} which produces  new slice knots into the setting of homology concordance.  

\begin{prop}\label{prop:annulus}
Let $R\subseteq M$ be a knot in a homology sphere.  Let $C$ be a concordance from $R$ to $R$ in the homology cobordism $W$ from $M$ to $M$.  Let $\eta_1,\eta_2$ be disjoint framed knots in $M-R$.  Suppose $\eta_1$ and $\eta_2$ cobound an annulus $A$ in $W-C$.  Choose a framing for $A$ and restrict it to obtain a framing on $\eta_1$ and $\eta_2$.  Then $R':=R_{\eta_1,\eta_2}(J,-J)$ is homology concordant to $R$.  

Moreover, if $F$ is a Seifert surface  for $R$ disjoint from $\eta_1$ and $\eta_2$ and $\alpha$ is a derivative for $R$ on $F$ then $R'$ has as a derivative $\alpha' = \alpha_{\eta_1,\eta_2}(K,-K)$.  If $W=M\times[0,1]$ and $C=R\times[0,1]$ then the concordance $C'$ satisfies that $\A(R')\to\A(C')\ot\A(R)$ are each isomorphisms and in $\A(C')$, $\alpha = \alpha'$.
\end{prop}
\begin{proof}
 Since all of the ideas appear in the proof of \cite[Theorem 3.1]{CD1}, we merely sketch the proof.    Let $N$ be a tubular neighborhood for $A$ and  $\Phi:S^1\times [0,1]\times B^2\to N$ be a framing for  $A$.  The restrictions $\Phi|_{S^1\times \{0\}\times B^2}$ and $\Phi|_{S^1\times \{1\}\times B^2}$ give framings for $\eta_1$ and $\eta_2$.   Let $p$ be a point in $\bdry B^2$.  One can build a new homology cobordism by starting with $W-N$ and gluing in a copy of $E(K)\times[0,1]$  so that the meridian of $K$ in $E(K)\times\{t\}$ is identified to the framed longitude of $A$, $\Phi[S^1\times \{t\}\times \{p\}]$ and so that the zero-framed longitude of $K$ is identified with the meridian of $A$, $\Phi[\{1\}\times \{t\}\times \bdry D]$.  Call the resulting manifold $W'$.  Since the meridian of $A$ generates $H_1(W-N)\cong \Z$ and is identified to the null-homologous longitude of $K$, a study of the Mayer-Veitoris sequence
$$
\dots \to H_*(\bdry N)\to H_*(W-N)\oplus H_*(E(K)\times [0,1])\to H_*(W')\to\dots
$$  reveals that $H_p(W')=0$ whenever $p>0$ so that $W'$ is a homology cobordism between its boundary components.

Since $\Phi$ restricts to the chosen framings of $\eta_1$ and $\eta_2$, $R\subseteq M\times\{1\}$ has now been replaced by $R'$.  $R\subseteq M\times\{0\}$ has not been changed in this process.  Since $N$ is disjoint from $R\times [0,1]$.  we now have a concordance (in the homology cobordism $W'$) from $R$ to $R'$.  Call this concordance $C'$.

Let $F$ be a Seifert surface for $R$ disjoint from $\eta_1$ and $\eta_2$.  Let $\alpha$ be a derivative of $R$ on $F$.  The satellite construction sends $F$ to a new surface $F'$ of the same genus bounded by $R_{\eta_1,\eta_2}(K,-K)$.  The link $\alpha$ is sent to $\alpha_{\eta_1,\eta_2}(K,-K)$.  According to Lemma~\ref{lem:genSatOpLinking} since $\alpha$ has zero linking numbers with its pushoffs, the same is true of $\alpha_{\eta_1,\eta_2}(K,-K)$.  Thus,  $\alpha_{\eta_1,\eta_2}(K,-K)$ is a derivative of $R'$.  

Suppose now that $W=M\times[0,1]$ and $C=R\times[0,1]$.  In particular then $\A(R\times\{0\})\to\A(C)\ot\A(R\times\{1\})$ are each isomorphisms.  Since $\eta_1$ and $\eta_2$ lift to the infinite cyclic cover of $M-R$, the annulus $A$ lifts to the infinite cyclic cover of $W-C$.  Notice that the infinite cyclic covers $\widetilde{W-C}$  and $\widetilde{W'-C'}$  differ by the replacement of the lifts of  $N$ with copies of $E(K)\times[0,1]$.  Since $N$ and $E(K)\times[0,1]$ have identical first homology, the same argument as was used to verify that $W'$ was a homology cobordism implies that this does not change the homology of the cover and 
$$\A(C) = H_1\left(\widetilde{W-C};\Q\right)\cong H_1\left(\widetilde{W'-C'};\Q\right) = \A(C')$$
This prove the isomorphism claim.   Notice that in $W-C=(M-R)\times[0,1]$, $\alpha\times\{0\}$ and $\alpha\times\{1\}$ cobound the annulus $\alpha\times[0,1]$ which lifts to the infinite cyclic cover of $C$.  This annulus might intersect $N$ in some number of meridonal disks for $N$.  To find a surface cobounded by $\alpha$ and $\alpha'$, replace each of these  disks by parallel Seifert surfaces for $K$, which also lift to $\widetilde{W'-C'}$.  Thus, $\alpha=\alpha'$ in $\A(C') = H_1\left(\widetilde{W'-C'};\Q\right)$.  This completes the proof.
\end{proof}

We will use this construction which preserves the concordance classes of a knot to alter derivatives.  In \cite{LiM} Livingston and Melvin compute the Blanchfield form of satellite knots.  The proposition below can be recovered as a consequence of their work.

\begin{prop}\label{prop:AlgConcInfect}
If $P$ is a knot in a homology sphere,  $P_\eta$ is a knotted doubling operator, and $\lnk(P,\eta) = c$ then for any knot $J$,  $P_\eta(J)$ is algebraically concordant to $P\# J_{(c,1)}$ 
\end{prop}
\begin{proof}

Since this result is a consequence of \cite[Theorem 2]{LiM} we merely sketch a proof in terms of the Seifert form.   Pick a Seifert surface $F$ for $P$ with intersects each $\eta$ in precisely $|c|$ points.  Notice then that the resulting Seifert surface $F'$ for $P_{\eta}(J)$ differs from $F$ by replacing neighborhoods each of of these $|c|$ points with a parallel copy of a Seifert surface for $J$  (or the mirror image if $c$ negative).  

It is now straightforward to compute the Seifert form on $F'$ and find that it is the same as the Seifert form for a surface bounded by $P\#J_{(c,1)}$.

\end{proof}

The definition of the solvable filtration (Definition~\ref{defn:n-sol}) makes no reference to the knot $K$ being in $S^3$.  Thus, it makes sense to ask if a knot in a homology sphere is $n$-solvable.  Moreover the proof of \cite[Theorem 8.9]{COT} passes through entirely in this setting and we have the following result.

\begin{prop}\label{prop:raisingsolvability} Suppose that $K$ is a knot in a homology $3$-sphere that admits a  Seifert surface, $F$, on which lies a derivative of link type $J$.
If $J$ is $n.5$-solvable  then $K$ is $(n+1.5)$-solvable and $J$ is associated to that $(n+1.5)$-solution.  

\end{prop}

\begin{proof}
 The idea is the same as that of the proof of \cite[Theorem 8.9]{COT}, which we merely summarize. Let $E$ be the fundamental cobordism between $M_K$ and $M_J$ constructed in Section~\ref{sect:bordism}. See also \cite[Subsection 8.1]{CHL7} and \cite[Lemma 2.5, proof of Lemma 2.3]{CHL3}.  Let $W$ be a $n.5$-solution for $M_J$. Let $Y=E\cup W$.  A study of the Mayer-Veitoris exact sequence
$$
H_*(M_J)\to H_*(E)\oplus H_*(W)\to H_*(Y)
$$
reveals that $Y$ is an $(n+1.5)$-solution.  Indeed the $(n+1)$-Lagrangians for $W$ become $(n+2)$-Lagrangians for $Y$ and the $n$-duals for $W$ become $(n+1)$-duals for $Y$.  
%
%
%

\end{proof}

Finally we have the tools needed to prove Theorem~\ref{thm:sufficiency}.

\begin{proof}[Proof of Theorem~\ref{thm:sufficiency}]
For the sake of generality, we will allow $K\subseteq M$ to be a knot in a homology $3$-sphere.  Assume $K$ has a genus one Seifert surface $F$ on which lies a derivative $\alpha$ of  knot type $J$.  Let $\Delta_K(t) = (mt-(m+1))((m+1)t-m)$.  Assume $m\neq 0, -1$ and that up to algebraic concordance $[J] = [T_{(m,1)}] - [T_{(m+1,1)}]$ for some knot $T$.  We need to show that $K$ is $1.5$-solvable.

   Let $\delta$ be a simple closed curve on $F$ which intersects $\alpha$ transversely in a single point.  See Figure~\ref{fig:example} for an example.  Let $\delta^+$ and $\delta^-$ be the positive and negative pushoffs of $\delta$.    Since $\delta$ and $\alpha$ intersect positively and transversely in a single point.  Then $\lnk(\delta^+,\alpha) = \lnk(\delta^-,\alpha) + 1$.  Let $\lambda = \lnk(\delta^-,\alpha)$ and $\ell = \lnk(\delta^+,\delta^-)$.  In terms of the basis $\{\alpha, \delta\}$ for $H_1(F)$, we see the Seifert matrix 
$A=\left[\begin{array}{cc}
0&\lambda\\\lambda+1&\ell
\end{array}\right]$. 
The Alexander polynomial of $K$ can be computed as $(\lambda t-(\lambda+1))((\lambda+1)t-\lambda)$.  Since we assumed $\Delta_K(t) = (mt-(m+1))((m+1)t-m)$, it follows that $\lambda = m$ or $\lambda=-m-1$.  We assume $\lambda = m$.  The proof when $\lambda = -m-1$ is the same.

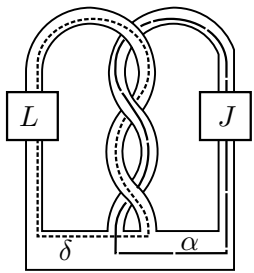
\begin{figure}[htbp]
\setlength{\unitlength}{1pt}
\begin{picture}(100,100)
\put(0,0){\includegraphics[height=100pt]{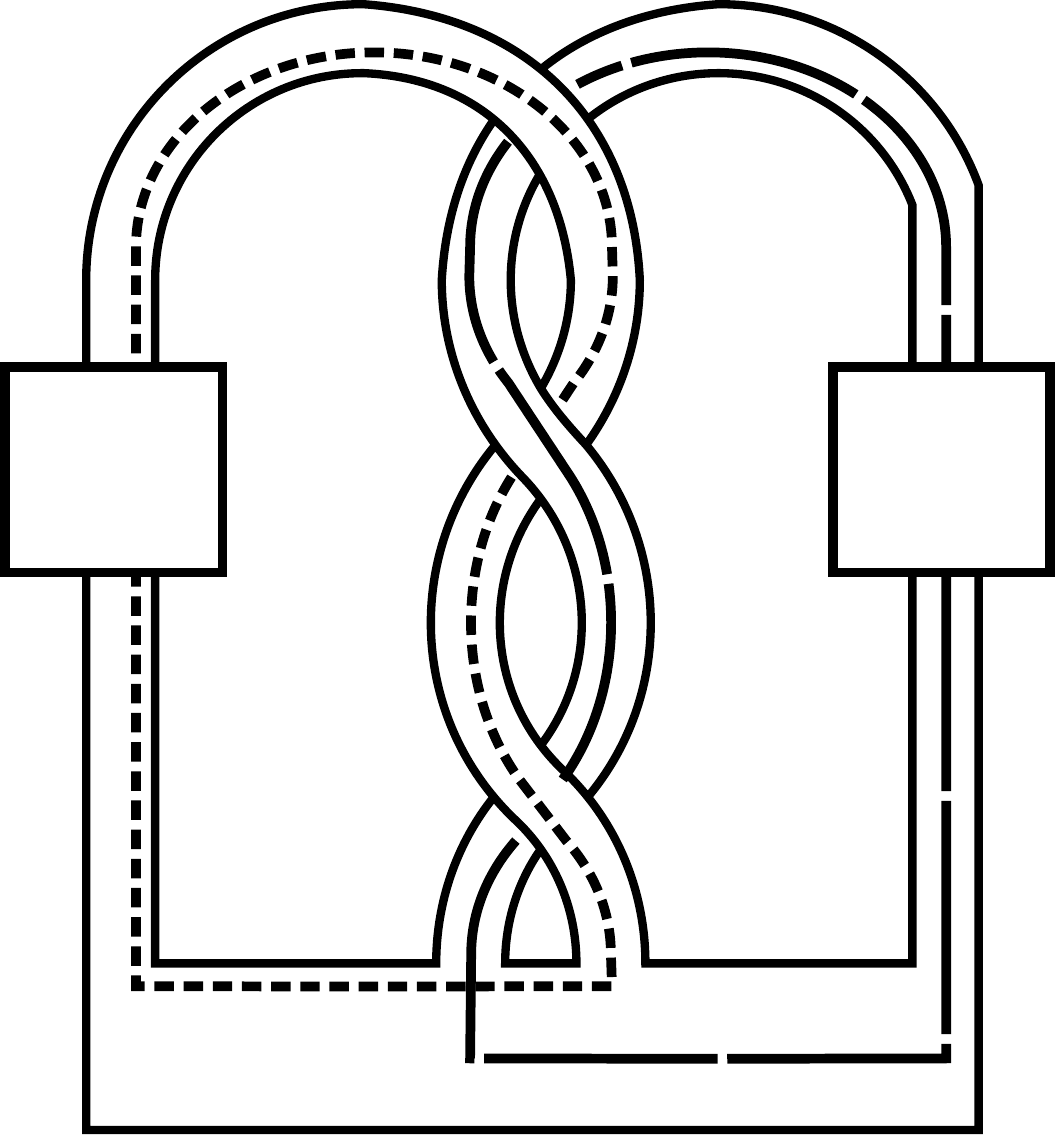}}
\put(80, 55){$J$}
\put(5, 55){$L$}
\put(20, 4){$\delta$}
\put(66, 8){$\alpha$}
\end{picture}
\caption{This genus one knot $K$ has a derivative $\alpha$ of knot type $J$.  By Proposition~\ref{prop:annulus} $K' = K_{\delta^+,\delta^-}(T,-T)$ is concordant to $K$ (in a homology cobordism).    If  $J_{\delta^+,\delta^-}(T,-T)$ is  algebraically slice then $K'$ (and so) $K$ is $1.5$-solvable.  }\label{fig:example}
\end{figure}

Now consider the the curves $\alpha$ and $\delta$ as sitting on the Seifert surface $F$.  Then $\delta^+$ and $\delta^-$ cobound an annulus in the complement of $K$.    Push this annulus into the interior of $M\times[0,1]$ pick a framing and restrict it to get a framing for $\delta^+$ and $\delta ^-$.    According to Proposition~\ref{prop:annulus} then $K$ is homology concordant to $K' = K_{\delta^+,\delta^-}(T, -T)$.  

Notice, additionally, that since $\delta^+$ and $\delta^-$ are disjoint from $F$, $K_{\delta^+,\delta^-}(T, -T)$ has a genus one Seifert surface on which $J_{\delta^+,\delta^-}(T, -T)$ is a derivative.  According to Proposition~\ref{prop:AlgConcInfect}, up to algebraic concordance
\begin{align}
[J_{\delta^+,\delta^-}(T, -T)] &= [J] + [T_{(m+1,1)}] + [(-T)_{(m,1)}]\nonumber
\\&=[T_{(m,1)}]-[T_{(m+1,1)}] + [T_{(m+1,1)}] - [T_{(m,1)}]&\nonumber
\\&=0.\nonumber
\end{align}
Then $K' = K_{\delta^+,\delta^-}(T, -T)$ has an algebraically slice derivative $\alpha' = \alpha_{\delta^+,\delta^-}(-T, T)$ and so is $1.5$-solvable via a $1.5$-solution associated to $\alpha'$.  

We now build a $1.5$-solution for $K'$.  Let $W$ be the $1.5$-solution for $K'$ with associated derivative $\alpha'$.  This means $\ker(\A(K')\to \A(W))$ is generated by $\alpha'$  Let $C$ be the concordance between $K$ and $K'$ and $E(C)$ be its exterior.  Glue $W$ and $E(C)$ together along the copies of $E(K')$ contained in their boundary.  It is straightforward to see that $W$ is a $1.5$-solution. Since $\alpha = \alpha'$ in $\A(C)$ we see that $\alpha$ generates $\ker(\A(K)\to \A(W'))$ and so is an associated derivative.  This completes the proof. 
 \end{proof}

\bibliographystyle{plain}
\bibliography{mybib9, biblio}

\begin{thebibliography}{10}

\bibitem{APS2}
M.~F. Atiyah, V.~K. Patodi, and I.~M. Singer.
\newblock Spectral asymmetry and {R}iemannian geometry. {II}.
\newblock {\em Math. Proc. Cambridge Philos. Soc.}, 78(3):405--432, 1975.

\bibitem{BGN}
Laurent Bartholdi, Rostislav Grigorchuk, and Volodymyr Nekrashevych.
\newblock From fractal groups to fractal sets.
\newblock In {\em Fractals in {G}raz 2001}, Trends Math., pages 25--118.
  Birkh\"auser, Basel, 2003.

\bibitem{CG1}
A.~J. Casson and C.~McA. Gordon.
\newblock On slice knots in dimension three.
\newblock In {\em Algebraic and geometric topology (Proc. Sympos. Pure Math.,
  Stanford Univ., Stanford, Calif., 1976), Part 2}, Proc. Sympos. Pure Math.,
  XXXII, pages 39--53. Amer. Math. Soc., Providence, R.I., 1978.

\bibitem{CG2}
A.~J. Casson and C.~McA. Gordon.
\newblock Cobordism of classical knots.
\newblock In {\em \`A la recherche de la topologie perdue}, volume~62 of {\em
  Progr. Math.}, pages 181--199. Birkh\"auser Boston, Boston, MA, 1986.
\newblock With an appendix by P. M. Gilmer.

\bibitem{CF}
Andrew Casson and Michael Freedman.
\newblock Atomic surgery problems.
\newblock In {\em Four-manifold theory (Durham, N.H., 1982)}, volume~35 of {\em
  Contemp. Math.}, pages 181--199. Amer. Math. Soc., Providence, RI, 1984.

\bibitem{CHL4}
Tim Cochran, Shelly Harvey, and Constance Leidy.
\newblock Link concordance and generalized doubling operators.
\newblock {\em Algebr. Geom. Topol.}, 8:1593--1646, 2008.

\bibitem{C}
Tim~D. Cochran.
\newblock Noncommutative knot theory.
\newblock {\em Algebr. Geom. Topol.}, 4:347--398, 2004.

\bibitem{CD1}
Tim~D. Cochran and Christopher~William Davis.
\newblock Counterexamples to {K}auffman's conjectures on slice knots.
\newblock {\em Adv. Math.}, 274:263--284, 2015.

\bibitem{CFT}
Tim~D. Cochran, Stefan Friedl, and Peter Teichner.
\newblock New constructions of slice links.
\newblock {\em Comment. Math. Helv.}, 84:617--638, 2009.

\bibitem{CHL3}
Tim~D. Cochran, Shelly Harvey, and Constance Leidy.
\newblock Knot concordance and higher-order {B}lanchfield duality.
\newblock {\em Geom. Topol.}, 13:1419--1482, 2009.

\bibitem{CHL7}
Tim~D. Cochran, Shelly Harvey, and Constance Leidy.
\newblock Derivatives of knots and second-order signatures.
\newblock {\em Algebr. Geom. Topol.}, 10(2):739--787, 2010.

\bibitem{CHL5}
Tim~D. Cochran, Shelly Harvey, and Constance Leidy.
\newblock Primary decomposition and the fractal nature of knot concordance.
\newblock {\em Math. Annalen}, 2010, (electronic).
\newblock DOI:10.1007/s00208-010-0604-5.

\bibitem{CHL6}
Tim~D. Cochran, Shelly Harvey, and Constance Leidy.
\newblock 2-torsion in the n-solvable filtration of the knot concordance group.
\newblock {\em Proc. London Math. Soc.}, 102(3):257--290, 2011.

\bibitem{COT}
Tim~D. Cochran, Kent~E. Orr, and Peter Teichner.
\newblock Knot concordance, {W}hitney towers and {$L\sp 2$}-signatures.
\newblock {\em Ann. of Math. (2)}, 157(2):433--519, 2003.

\bibitem{CooperThesis}
Daryl Cooper.
\newblock {\em Signatures of surfaces in 3-manifolds with applications to knot
  and link cobordism}.
\newblock PhD thesis, Warwick University, 1982.

\bibitem{Davis14}
Christopher~William Davis.
\newblock Linear independence of knots arising from iterated infection without
  the use of {T}ristram-{L}evine signature.
\newblock {\em Int. Math. Res. Not. IMRN}, (7):1973--2005, 2014.

\bibitem{FoMi}
Ralph~H. Fox and John~W. Milnor.
\newblock Singularities of {$2$}-spheres in {$4$}-space and cobordism of knots.
\newblock {\em Osaka J. Math.}, 3:257--267, 1966.

\bibitem{GL2}
Patrick~M. Gilmer and Charles Livingston.
\newblock On surgery curves for genus-one slice knots.
\newblock {\em Pacific J. Math.}, 265(2):405--425, 2013.

\bibitem{HedKirk}
Matthew Hedden and Paul Kirk.
\newblock Instantons, concordance, and {W}hitehead doubling.
\newblock {\em J. Differential Geom.}, 91(2):281--319, 2012.

\bibitem{OnKnots}
Louis~H. Kauffman.
\newblock {\em On knots}, volume 115 of {\em Annals of Mathematics Studies}.
\newblock Princeton University Press, Princeton, NJ, 1987.

\bibitem{Kirbyproblemslist}
Robion Kirby.
\newblock Problems in low dimensional topology.

\bibitem{L5}
J.~Levine.
\newblock Knot cobordism groups in codimension two.
\newblock {\em Comment. Math. Helv.}, 44:229--244, 1969.

\bibitem{L6}
J.~P. Levine.
\newblock Link invariants via the eta invariant.
\newblock {\em Comment. Math. Helv.}, 69(1):82--119, 1994.

\bibitem{Lith1}
R.~A. Litherland.
\newblock Cobordism of satellite knots.
\newblock In {\em Four-manifold theory (Durham, N.H., 1982)}, volume~35 of {\em
  Contemp. Math.}, pages 327--362. Amer. Math. Soc., Providence, RI, 1984.

\bibitem{LiM}
Charles Livingston and Paul Melvin.
\newblock Abelian invariants of satellite knots.
\newblock In {\em Geometry and topology (College Park, Md., 1983/84)}, volume
  1167 of {\em Lecture Notes in Math.}, pages 217--227. Springer, Berlin, 1985.

\bibitem{R}
Dale Rolfsen.
\newblock {\em Knots and links}, volume~7 of {\em Mathematics Lecture Series}.
\newblock Publish or Perish Inc., Houston, TX, 1990.
\newblock Corrected reprint of the 1976 original.

\bibitem{Rudin}
Walter Rudin.
\newblock {\em Real and complex analysis}.
\newblock McGraw-Hill Book Co., New York, third edition, 1987.

\bibitem{ste}
Bo~Stenstr{\"o}m.
\newblock {\em Rings of quotients}.
\newblock Springer-Verlag, New York, 1975.
\newblock Die Grundlehren der Mathematischen Wissenschaften, Band 217, An
  introduction to methods of ring theory.

\bibitem{Str}
Ralph Strebel.
\newblock Homological methods applied to the derived series of groups.
\newblock {\em Comment. Math. Helv.}, 49:302--332, 1974.

\end{thebibliography}
\end{document}